\documentclass{amsart}

\usepackage{amsmath,amsthm,amsfonts,amscd,amssymb} 
\usepackage{verbatim}      	% Allows quoting source with commands.
\usepackage{epsfig}         	% Allows inclusion of eps files.
\usepackage{url}		% Allows good typesetting of web URLs.
\usepackage{epic,eepic,latexsym}
\usepackage{graphicx}
\usepackage{color}
\usepackage{lcd}
\usepackage{mathrsfs}
\usepackage[centering,text={15.5cm,22cm},
        marginparwidth=20mm,dvips]{geometry}
\usepackage{verbatim}
\usepackage[linktocpage]{hyperref}
\usepackage{lscape}
\usepackage{tabularx}
\usepackage{marginnote}
\usepackage[all]{xy}
\usepackage{enumitem}

\DeclareMathOperator{\Pic}{Pic}

\DeclareMathOperator{\Spec}{Spec}
\DeclareMathOperator{\id}{Id}

\DeclareMathOperator{\NE}{NE}
\DeclareMathOperator{\dist}{dist}
\DeclareMathOperator{\Tr}{Tr}

\DeclareMathOperator{\Limits}{Limits}
\DeclareMathOperator{\PIN}{PIN}
\DeclareMathOperator{\NIN}{NIN}
\DeclareMathOperator{\IN}{IN}
\DeclareMathOperator{\val}{val}
\DeclareMathOperator{\Val}{{\bf val}}

\DeclareMathOperator{\SL}{SL}
\DeclareMathOperator{\trop}{trop}

\DeclareMathOperator{\Conv}{{\bf Conv}}

\DeclareMathOperator{\Newt}{Newt}
\DeclareMathOperator{\Residue}{Res}

\DeclareMathOperator{\Spf}{Spf}
\DeclareMathOperator{\PGL}{PGL}

\let\?=\overline
\let\:=\colon
\let\bb=\mathbb

\let\rar=\rightarrow
\let\f=\mathfrak
\let\s=\mathcal

\let\wh=\widehat
\let\wt=\widetilde
\let\mb=\mbox

\newcommand {\kk} {\Bbbk}

\title{Tropical theta functions and log Calabi-Yau surfaces}
\author{Travis Mandel}
\address{University of Utah\\
Department of Mathematics\\
155 S 1400 E RM 233\\
Salt Lake City, UT, 84112-0090}
\email{mandel{\char'100}math.utah.edu}

\theoremstyle{plain}% default
 \newtheorem{thm}{Theorem}[section]
 
 \newtheorem{lem}[thm]{Lemma}
 \newtheorem{lemdfn}[thm]{Lemma/Definition}
 \newtheorem{prop}[thm]{Proposition}
 \newtheorem{conj}[thm]{Conjecture}
 \newtheorem{cor}[thm]{Corollary}

\theoremstyle{definition}
 \newtheorem{dfn}[thm]{Definition}
  \newtheorem{dfns}[thm]{Definitions}
 \newtheorem{ntn}[thm]{Notation}
  \newtheorem{convention}[thm]{Convention}
 \newtheorem{eg}[thm]{Example}
 \newtheorem{egs}[thm]{Examples}

\theoremstyle{remark} 
 \newtheorem{rmk}[thm]{Remark}

\begin{document}

\maketitle           

\begin{abstract}
We generalize the standard combinatorial techniques of toric geometry to the study of log Calabi-Yau surfaces.    The character and cocharacter lattices are replaced by certain integral linear manifolds described in \cite{GHK1}, and monomials on toric varieties are replaced with the canonical theta functions defined in \cite{GHK1} using ideas from mirror symmetry.  We describe the tropicalizations of theta functions and use them to generalize the dual pairing between the character and cocharacter lattices.  We use this to describe generalizations of dual cones, Newton and polar polytopes, Minkowski sums, and finite Fourier series expansions.  We hope that these techniques will generalize to higher-rank cluster varieties.
\end{abstract}

\setcounter{tocdepth}{1}
\tableofcontents  

\section{Introduction}

The main goal behind this paper is to use ideas from mirror symmetry to generalize the powerful techniques of toric geometry to log Calabi-Yau varieties---those admitting a holomorphic volume form with simple poles along the boundary divisors of certain compactifications.  This class of varieties includes several highly studied objects, with one of the simplest classes of log Calabi-Yau varieties, i.e., cluster varieties (and certain partial compactifications thereof), including such objects as character varieties, flag manifolds, and semisimple groups.  We will focus here on log Calabi-Yau surfaces, which are roughly the same as the fibers of rank $2$ cluster $\s{X}$-varieties (cf. \cite{GHK3}).

\subsection{Some Main Results}\label{Results}

We assume throughout that our log Calabi-Yau surface $U$ is ``positive,'' unless otherwise stated.  See \S \ref{IntroSetup} for this and several other relevant definitions.

Toric varieties are typically understood by studying their character and cocharacter lattices, denoted $M$ and $N$, respectively.  \cite{GHK1} generalizes the cocharacter lattice by defining the tropicalization $U^{\trop}$ of a log Calabi-Yau surface $U$.  $U^{\trop}$ is an integral linear manifold (cf. \S \ref{ils}), and the integral points $q\in U^{\trop}(\bb{Z})$ correspond to boundary divisors $D_q$ for certain compactifications of $U$. \cite{GHK1} then uses toric degenerations, modified by scattering diagrams, to construct a 
 mirror family $\s{V}\rar \Spec B$ of log Calabi-Yau surfaces, with $U^{\trop}(\bb{Z})$ serving as a generalization of the character lattice for $\s{V}$.    That is, points $q\in U^{\trop}(\bb{Z})$ correspond to canonical ``theta functions'' $\vartheta_q$ forming a $B$-module basis for $A:=H^0(\s{V},\s{O}_{\s{V}})$; i.e., $A = \bigoplus_{q\in U^{\trop}(\bb{Z})} B \cdot \vartheta_q$.

Let $V$ be a generic fiber of the mirror.  \cite{GHK2} shows that $V$ is deformation equivalent to $U$, and so in particular, we can consider the tropicalization $V^{\trop}$ whose integral points correspond to boundary divisors of compactifications of $V$.  For $f$ a regular function on $V$ and $v\in V^{\trop}(\bb{Z})$, we define $f^{\trop}(v) := \val_{D_v}(f)$.  Thus, we can define a pairing $\langle\cdot,\cdot\rangle:U^{\trop}(\bb{Z})\times V^{\trop}(\bb{Z}) \rar \bb{Z}$ by $\langle q,v\rangle : = \vartheta_q^{\trop}(v)$.  This pairing extends continuously and equivariantly under scaling to a pairing $U^{\trop}\times V^{\trop}\rar \bb{R}$, also denoted $\langle \cdot,\cdot\rangle$.  One should view $\langle \cdot,\cdot \rangle$ as a generalization of the dual pairing between $M_{\bb{R}}$ and $N_{\bb{R}}$ in the toric situation.

In \S \ref{App} we show how to use $\langle \cdot,\cdot \rangle$ to preform the basic constructions of toric geometry, such as defining dual cones (as in the construction of toric varieties from fans) and Newton and polar polytopes.  
 We show that these polytopes have the same geometric interpretations as in the toric situation.  For example, integral points in the polytope correspond to global sections of an associated line bundle, and convexity corresponds to ampleness of this line bundle (cf. Corollary \ref{AmpleConvex}).

As an application, we prove certain orthogonality properties of a canonical pairing on $A$.  Briefly, \cite{GHK2} defines a canonical homology class $\gamma$ in $\s{V}$ (the class of a conjectural SYZ fibration) and shows that the trace pairing $\Tr(f,g):=\int_{\gamma} fg\Omega$ is non-degenerate.  $\Omega$ here is the unique holomorphic volume form on $V$ with simple poles along the boundary such that $\int_{\gamma} \Omega=1$.  Equivalently, $\Tr(f,g)$ is the coefficient of $\vartheta_0:=1$ in the theta function expansion of $fg$.  $\Tr$ thus makes $A$ into a Frobenius algebra.

\S 0.4 of \cite{GHK1} conjectures that $\Tr(\vartheta_{q_1},\vartheta_{q_2})$ is given by a certain log Gromov-Witten count of curves.  In particular, although $\Tr(\vartheta_q,\vartheta_0)=\delta_{q,0}$, the theta functions certainly do not form a orthogonal basis with respect to $\Tr$.  However, V.V. Fock made the remarkable conjecture that something similar does hold: he predicted that one does have $\Tr_q(\vartheta_p):=Tr(\vartheta_p,\vartheta_q^{-1})=\delta_{p,q}$.  This turns out to be false in general, but in \S \ref{period} we give the following general collection of conditions in which this relationship does hold:
\begin{thm}[\ref{Fourier}]\label{FourierIntro}
Let $f=\sum_q c_{q} \vartheta_q$ be a function on $V$.  Suppose that at least one of the following holds:
\begin{itemize}[noitemsep] 
\item $r$ is not in the ``strong convex hull'' of any point $q\in \Newt(f)\cap U^{\trop}(\bb{Z})$, except possibly $q=r$.  In particular, this includes cases where
 $r$ is a vertex of $\Newt(f)$, as well as cases where $r$ is in the complement of $\Newt(f)$.
\item $r\in U^{\trop}(\bb{Z})$ is in the cluster complex (i.e., $r=0$ or $\langle r,v\rangle >0$ for some $v\in V^{\trop}$). 
\end{itemize}
 Then $c_{r}=\Tr_r(f)$.
In particular, if every point of $\Newt(f)\cap U^{\trop}(\bb{Z})$ which is not a vertex is in the cluster complex, then
\begin{align}\label{FSeriesIntro}
f=\sum_{r\in U^{\trop}(\bb{Z})} \Tr_r (f) \vartheta_r.
\end{align}
\end{thm}
The proof for the first condition is based on the residue theorem and the relationship between strong convex hulls and the zeroes and poles of theta functions.  The proof for the second condition follows from reducing to the toric case.

One may think of Equation \ref{FSeriesIntro} as a generalization of the formula for Fourier series expansions.  Indeed, the usual formula for (finite) Fourier expansions follows from applying the theorem to the case where $V$ is toric and then restricting to the orbits of the torus action.

As another application, we prove some Minkowski sum formulas for functions in $A$.  The Newton polytope of a function $f$ indicates which theta functions might show up in the theta function expansion of $f$, and Minkowski sums allow one to describe the Newton polytope of a product of functions. 
  More precisely, the Minkowski sum $\Newt(f)+\Newt(g)$ is defined to be $\Newt(fg)$.  Minkowski sums may therefore be viewed as a tropicalized version of multiplication.  $U^{\trop}$ contains a singular point that prevents addition from being defined as easily as in the toric case.  However, $U^{\trop}$ is covered by convex cones, and addition does of course make sense when restricting to these cones.

\begin{thm}[\ref{MinkowskiSum}]\label{MinkowskiSumIntro}
The Minkowski sum of a collection $Q_1,\ldots,Q_s$ of integral polytopes containing the origin is given by 
\begin{align*}
\Conv\left(\bigcup_{\sigma} \left\{ q_1+_{\sigma} \ldots +_{\sigma} q_s \left| q_i \in Q_i \cap \sigma \right. \right\}\right),
\end{align*} where the union is over all convex cones $\sigma$ in $U^{\trop}$, and $+_\sigma$ denotes addition as defined in $\sigma$.
\end{thm}
In fact, we will see that only finitely many convex cones and $s$-tuples are needed, so Minkowski sums really are computable.  We will also give two other version of the above theorem: one says that we can take Minkowski sums by working on the universal cover of $U^{\trop}$ and doing addition in a very natural way.  The other applies only to ``finite type'' cases (in the cluster sense), and says that the Minkowski sum can be computed by taking the unions of the Minkowski sums with respect to each seed.  \cite{Shen} proved this version for cluster varieties of type $A_n$.

We also prove several properties of the pairing $\langle \cdot,\cdot \rangle$.  For example, we prove that it satisfies the following generalization of bilinearity: call a function on $U^{\trop}$ or $V^{\trop}$ tropical if it is integral, piecewise-linear, and convex along broken lines.\footnote{Convexity along broken lines is a notion from \cite{GHKK} which we show is, in our situation, equivalent to \cite{FG1}'s notion of ``convex with respect to every seed.''}  The tropical functions form a min-plus algebra, and we call a tropical function $\varphi$ {\it indecomposable} if it cannot be written as a minimum of two other tropical functions, neither of which is $\varphi$.  The tropical functions generalize convex integral piecewise-linear functions on $N_{\bb{R}}$ and $M_{\bb{R}}$, and the indecomposable functions generalize the linear functions.  \cite{GHKK} conjectures that tropicalizations of regular functions are tropical for {\it any} log Calabi-Yau variety, and \cite{FG1} conjectures that the theta functions---not just their tropicalizations---satisfy a related indecomposability condition (now known to be false in general).  For the log Calabi-Yau surface cases, we show:

\begin{thm}[\ref{IndeTropical}]\label{IndeTropicalIntro}
 The tropical functions are exactly the tropicalizations of regular functions, and the indecomposable tropical functions are exactly the tropicalizations of theta functions.
\end{thm}
Thus, we may say that the pairing $\langle \cdot,\cdot\rangle$ is ``integral bi-indecomposable-tropical,'' meaning that if we fix either entry to be some integral point, then the pairing is an indecomposable tropical function in the other entry.  This generalizes the (integral) bilinearity of the usual dual pairing.  See Remark \ref{NegativeTropical} for an extension of Theorem \ref{IndeTropicalIntro} to the non-positive cases.

\subsection{Setup}\label{IntroSetup}

Throughout this paper, $Y$ will denote a smooth, projective, rational surface over an algebraically closed field $\kk$ of characteristic $0$.  The {\it boundary} $D$ is a choice of nodal anti-canonical divisor in $Y$, and $U$ will denote $Y\setminus D$.  Here, $D=D_1+\ldots D_n$ is a either a cycle of smooth irreducible rational curves $D_i$ with normal crossings, or if $n=1$, $D$ is an irreducible curve with one node.  By a {\it compactification} of $U$, we mean such a pair $(Y,D)$ (\cite{GHK_MLP} calls these ``compactifications with maximal boundary'').  We call $(Y,D)$ a {\it Looijenga pair}, as in \cite{GHK1}, and we call $U$ a {\it log Calabi-Yau surface} or a {\it Looijenga interior}.

For a Looijenga pair $(Y,D)$, we define a {\it toric blowup} to be a Looijenga pair $(\wt{Y},\wt{D})$ together with a birational map $\wt{Y}\rar Y$ which is a  blowup at a nodal point of the boundary $D$, such that $\wt{D}$ is the preimage of $D$.  Note that taking a toric blowup does not change the interior $U = Y\setminus D = \wt{Y}\setminus \wt{D}$.  We also use the term toric blowup to refer to finite sequences of such blowups.

By a {\it non-toric blowup} $(\wt{Y},\wt{D})\rar (Y,D)$, we will always mean a blowup $\wt{Y}\rar Y$ at a non-nodal point of the boundary $D$ such that $\wt{D}$ is the proper transform of $D$.  Let $(\?{Y},\?{D})$ be a Looijenga pair where $\?{Y}$ is a toric variety and $\?{D}$ is the toric boundary.  We say that a birational map $Y \rar \?{Y}$ is a {\it toric model} of $(Y,D)$ (or of $U$) if it is a finite sequence of non-toric blowups.  Every Looijenga pair has a toric blowup which admits a toric model (\cite{GHK1}, Prop. 1.19).

In the language of cluster varieties, toric model corresponds to choices of {\it seeds} for cluster structures on $U$.  We will therefore use the term ``seed'' interchangably with the term ``toric model.''

According to \cite{GHK2}, all deformations of $U$ come from sliding the non-toric blowup points along the divisors $\?{D}_i\subset D$ without ever moving them to the nodes of $D$.  We call $U$ {\it positive} if some deformation of $U$ is affine.  This is equivalent to saying that $D$ supports an effective $D$-ample divisor, meaning a divisor whose intersection with each component of $D$ is positive.  We will always take the term $D$-ample to imply effective, unless otherwise stated.  
 We will assume that $U$ is positive throughout Sections 3-6, unless otherwise stated.

\subsection*{Outline of the Paper}

\subsection{The Tropicalization of $U$}\label{Betas} In \S \ref{tropu}, we review \cite{GHK1}'s construction of the tropicalization of $U$, an integral linear manifold denoted $U^{\trop}$.  The integral points $U^{\trop}(\bb{Z})\subset U^{\trop}$ generalize the cocharacter lattice $N$ for toric varieties.  If $q\in U^{\trop}(\bb{Z})$ is primitive (i.e., nonzero and not a positive integral multiple of some other element of $U^{\trop}(\bb{Z})$), then it corresponds to an irreducible divisor $D_q$ in the boundary of some compactification of $U$.  If $q$ is a multiple $|q|\in \bb{Z}_{\geq 0}$ times a primitive element, then the corresponding divisor is $|q|D_q$.  We call $|q|$ the {\it index} of $q$.  

$U^{\trop}$ is homeomorphic to $\bb{R}^2$, but it comes with an integral linear structure (singular at the origin) that captures the intersection data of the boundary divisors. 
 We analyze the integral piecewise-linear functions on $U^{\trop}$ using the intersection theory on compactifications of $U$: an integral piecewise-linear functions $\varphi$ on $U^{\trop}$ corresponds to a Weil divisor $W_{\varphi}:=\sum \varphi(v_i) D_{v_i}$ on a compactification $(Y,D=\sum D_{v_i})$ of $U$, and the ``bending parameter'' of $\varphi$ across $\rho_{v_i}$ is the intersection number $W_{\varphi}\cdot D_{v_i}$.  Let $\beta_{v_1,\ldots,v_s}$ denote the set of functions which have bending parameter $|v_i|$ along the ray $\rho_{v_i}$ generated by $v_i$, for each $i$, and otherwise has no other bends.  As a consequence of the symmetry of the intersection product, we find:
\begin{prop}[\ref{BetaSymmetry}]
If the intersection matrix $H=(D_i\cdot D_j)$ for some compactification of $U$ is invertible, then $\beta_v$ consists of a single function for each $v$, and $\beta_{v}(w) = \beta_{w}(v)$ for all $v,w\in U^{\trop}(\bb{Q})$.
\end{prop}
This Proposition is in fact closely related to the symmetry in Theorem \ref{symdualIntro} below.  At the end of \S \ref{tropu}, we give the definitions of lines and polygons in $U^{\trop}$, as introduced in \cite{GHK2}.

\subsection{Constructing the Mirror and the Theta Functions} In \S \ref{mirror} we review \cite{GHK1}'s construction of the mirror family $\s{V}$ of $U$.  The theta functions $\vartheta_q$, $q\in U^{\trop}(\bb{Z})$, are defined in terms of broken lines, which are certain piecewise-straight lines in $U^{\trop}$ with attached monomials.

In \S \ref{SeedScatter} we describe an alternate construction of $U^{\trop}$ using scattering diagrams and broken lines.  This explains the relationship between the canonical integral linear struture on $U^{\trop}$ and the vector space structures corresponding to various seeds.  In \S \ref{clustercomplex}, we describe a particularly nice part of the scattering diagram called the ``cluster complex'' (technically the intersection of \cite{FG1}'s cluster complex with $U^{\trop}$, cf. Proposition 4.3 of \cite{Man2}), and we show that the theta functions corresponding to integral points in the cluster complex are cluster monomials (i.e., they each restrict to a monomial on some seed torus).  Since the initial posting of this paper, \cite{GHKK} has proven this for arbitrary cluster varieties.

At the end of \S \ref{mirror}, we review \cite{GHK2}'s construction of compactifications of $\s{V}$.

\subsection{Theta Functions and their Tropicalizations}
In \S \ref{trop}, we explicitely describe the tropicalizations of theta functions, as defined above in \S \ref{Results}, and we investigate some of their properties.  We begin by describing a way to identify $U^{\trop}$ with $V^{\trop}$ for computational purposes (analogous to using the standard inner product to identify $N_{\bb{R}}$ with $M_{\bb{R}}$ in the toric situation).  We find an explicit description of $\langle \cdot,\cdot\rangle$ in \S \ref{ValFun} and \S \ref{ttf}.  For example, as investigated in \S \ref{negbend}, tropical theta functions which are negative everywhere bend along at most a single ray.  On the other hand, each seed 
 induces a different integral linear structure on $U^{\trop}$, and the tropical theta functions which are positive somewhere are linear with respect to some seed.  See Corollary \ref{level} and Proposition \ref{levelpos} for explicit descriptions of the fibers of these tropical theta functions.

In \S \ref{sympair}, we use these explicit descriptions to prove the following symmetry of $\langle \cdot,\cdot \rangle$: since $V$ is itself log Calabi-Yau, we can choose a compactification and construct a family $\s{U}$ mirror to $V$.  We describe how to identify $U^{\trop}$ with the tropicalization of a fiber of $\s{U}$ (in a way which is in fact induced by an identification of $U$ with a fiber of $\s{U}$), and note that this allows one to define a second, a priori different pairing between $U^{\trop}$ and $V^{\trop}$, given by $\langle q,v\rangle^{\vee}:= \val_{D_q}(\vartheta_v)$ (that is, we have switched the roles of $U$ and $V$).
\begin{thm}[\ref{symdual}]\label{symdualIntro}
The two pairings $\langle\cdot,\cdot\rangle$ and $\langle \cdot,\cdot \rangle^{\vee}$ are in fact the same.
\end{thm}
A generalization of this for cluster varieties has been conjectured by \cite{GHKK}.

\S \ref{Convexity} introduces the tropical functions mentioned above in \S \ref{Results}.  Convexity along a broken line locally means convexity with respect to a linear structure in which the broken line is straight.  Tropical functions are defined to be convex along all broken lines, and we show that this is equivalent (for globally defined piecewise-linear functions) to being convex with respect to the linear structure induced by each seed.  We then prove Theorem \ref{IndeTropicalIntro} and make several conjectures about how this might generalize to higher dimensional cluster varieties.

\subsection{Toric Constructions for Log Calabi-Yau's}

In \S \ref{App} we use the pairing $\langle \cdot,\cdot\rangle$ to generalize several constructions from toric geometry.  \S \ref{BundlePolygons} focuses on constructions involving polytopes.  For example, we define the {\it strong convex hull} of a set $Q\subset U^{\trop}$ as
\begin{align*}
\Conv(Q) = \left\{ x \in U^{\trop}|\langle x,v\rangle \geq \inf_{q\in Q} \langle q,v\rangle \mbox{ for all } v\in V^{\trop}\right\}.
\end{align*}
We call a polytope strongly convex if it equals its own strong convex hull.  Such polytopes and their Minkowski sums also appear in the literature on cluster varieties (cf. \cite{FG2} and \cite{Shen}).
 We show:
\begin{thm}[\ref{StrongConvexLines}]
A rational polytope $Q$ is strongly convex if and only if any broken line segment with endpoints in $Q$ is entirely contained in $Q$.
\end{thm}

Consider a regular function $f:=\sum_{q\in Q} a_q \vartheta_q$, $Q\subset U^{\trop}(\bb{Z})$, $a_q\neq 0$.  The {\it Newton polytope} of $f$ is defined to be $\Conv(Q)$.  On the other hand, a Weil divisor $W$ supported on the boundary of a compactification of $V$ corresponds to a piecewise-linear function $\varphi_W$ on $V^{\trop}$, hence to a polytope $\Delta_W:=\{\varphi_W \leq 1\}$ in $V^{\trop}$.  $\Delta_W^{\vee}\subset U^{\trop}$ is then defined to be the Newton polytope of a generic section of $\s{O}(W)$, and if $W$ is effective, this agrees with the polar polytope
\begin{align*}
\Delta_W^{\circ}:=\{q\in U^{\trop}|\langle q,v\rangle \geq -1 \mbox{ for all } v\in \Delta_W\}.
\end{align*}  The theta functions corresponding to integral points in $\Delta_W^{\vee}$ form a canonical basis of global sections for $\s{O}(W)$.  This relationship was previously examined in \cite{GHK2} for $W$ strictly effective (i.e., for $\varphi_W \geq 0$, or for $\Delta_W^{\vee}$ containing the origin in its interior).

Other properties of polytopes from the toric situation now easily generalize.  For example, we find exactly as in the toric situation that the number of lattice points on edges of $\Delta_W^{\vee}$ is related to certain intersection numbers of $W$ with the boundary divisors (cf. Proposition \ref{PointsIntersection}).  

In \S \ref{DualCones} we note that the notion of {\it dual cones} also generalizes from toric varieties: the dual to a cone $\sigma\subset V^{\trop}$ is the cone 
\begin{align*}
\sigma^{\vee}:=\{ q \in U^{\trop}|\langle q,v\rangle \geq 0 \mbox{ for all } v\in \sigma\}.
\end{align*}
If $\sigma^{\vee}$ is two-dimensional, then $\Spec$ of the ring generated by the $\vartheta_q$'s with $q\in \sigma^{\vee}$ is obtained from $V$ by gluing boundary divisors corresponding to the boundary rays of $\sigma$ and then contracting the $(-1)$-curves which intersect these boundary divisors (see Proposition \ref{ConeAffine}).

In \S \ref{Minkowski} we introduce the Minkowski sums mentioned above in \S \ref{Results}, and we prove Theorem \ref{MinkowskiSumIntro}, along with the two other Minkowski sum formulas mentioned above.  The key idea behind the proofs is that the only broken lines which contribute to the tropicalization are actually straight lines in $U^{\trop}$.

\subsection{Relation to Cluster Varieties}\label{ClusterRelation}
\cite{FG1} defines certain varieties, called {\it cluster varieties}, constructed by gluing together algebraic tori in a certain combinatorial way.  \cite{GHK3} interpreted this gluing geometrically and showed that generic Looijenga interiors can be identified, up to codimension $2$, with fibers of certain cluster $\s{X}$ varieties.  As previously mentioned, what \cite{FG1} calls a {\it seed} has roughly the same data as that of a toric model for $U$.  \cite{FG1} defines tropicalizations $\s{A}^{\trop}$ and $\s{X}^{\trop}$ of their cluster $\s{A}$ and $\s{X}$ varieties, and \cite{GHK1}'s $U^{\trop}$ can be identified with a certain fiber of $\s{X}^{\trop}$ (this fiber is the image of a canonical map from $\s{A}^{\trop}$ to $\s{X}^{\trop}$).  What we call the {\it cluster complex} in $U^{\trop}$ is really the intersection of \cite{FG1}'s cluster complex (a certain subset of $\s{X}^{\trop}$) with $U^{\trop}$.  See \cite{Man2} for a more detailed summary of \cite{GHK3} and the relationship between cluster varieties and Looijenga interiors.

\subsection{Acknowledgements}
This paper is based on part of the author's thesis, which was written while in graduate school at the University of Texas at Austin.  I would like to thank my advisor, Sean Keel, for introducing me to this topic and for all his suggestions, insights, and support, and also for allowing me access to preliminary versions of his papers with Paul Hacking and Mark Gross.  I also want to thank Andy Neitzke, James Pascaleff, Yuan Yao, and Yuecheng Zhu for many valuable conversations, as well as the referee for numerous helpful suggestions.  The author was partially supported by the Center of Excellence Grant ``Centre for Quantum Geometry of Moduli Spaces'' from the Danish National Research Foundation (DNRF95).

\section{The Tropicalization of U} \label{tropu} 
This section examines $U^{\trop}$ with its integral linear structure as defined in \cite{GHK1}. 
 $U^{\trop}$ is a natural generalization of the cocharacter plane $N_{\bb{R}}$ corresponding to a toric surface, and the relationship between $U^{\trop}$ and the mirror is a natural generalization of the character plane $M_{\bb{R}}$.  We do not require $U$ to be positive in this section.  

\subsection{Some Generalities on Integral Linear Stuctures} \label{ils}

A manifold $B$ is said to be {\it (oriented) integral linear} if it admits charts to $\bb{R}^n$ which have transition maps in $\SL_n(\bb{Z})$.  We allow $B$ to have a set $O$ of singular points of codimension at least $2$, meaning that these integral linear charts only cover $B':=B\setminus O$.   $B'$ has a canonical set of {\it integral points} which come from using the charts to pull back $\bb{Z}^n\subset \bb{R}^n$.  Our space of interest, $B=U^{\trop}$, will be homeomorphic to $\bb{R}^2$ and will typically have a singular point at $0$ (which we say is also an integral point).

$B'$ admits a flat affine connection, defined using the charts to pull back the standard flat connection on $\bb{R}^n$.  Furthermore, pulling back along these charts give a local system $\Lambda$ of integral tangent vectors on $B'$, along with a dual local system $\Lambda^*$ in the cotangent bundle.  
 Note that the monodromy $\mu$ of $\Lambda$ is contained in $SL_n(\bb{Z})$, 
 so the wedge form on any exterior product $\Lambda^{\bullet} TB$ 
 commutes with parallel transport.  For $B$ two-dimensional, we will often use $\wedge$ to denote the canonical skew-form on $TB'$.

Note that a chart $\psi$ with a connected set $\sigma$ in its domain induces an embedding of $\sigma$ into $T_pB'$ for any $p \in \sigma$, commuting with parallel transport in $\sigma$ and taking integral points to $\Lambda$.  When we talk about addition, scalar multiplication, or wedge products of points on $\sigma$, we will mean the corresponding operations induced by this identification with the tangent space (equivalently, induced by the identification with $\psi(\sigma)\subset \bb{R}^n$).  Because of the monodromy, these operations do depend on the choice of $\sigma$, but not on the specific choice of chart.

\subsubsection{Integral Linear Functions}\label{int lin fun}

By a {\it linear map} $\varphi:B_1\rar B_2$ of integral linear manifolds, we mean a continuous map such that for each pair of integral linear charts $\psi_i:U_i\rar \bb{R}^n$, $U_i\subset B_i'$ with $\varphi(U_1)\subset U_2$, we have that $\psi_2\circ \varphi\circ \psi_1^{-1}$ is linear in the usual sense.  $\varphi$ is {\it integral linear} if it also takes integral points to integral points.

Fix a finite rank lattice $P^{gp}$.  $P^{gp}_{\bb{R}}:=P^{gp}\otimes_{\bb{Z}} \bb{R}$ has an obvious integral linear structure with $P^{gp}$ as the integral points.  By a $P^{gp}$-valued {\it integral linear function}, we will mean an integral linear map to $P^{gp}_{\bb{R}}$.  We can thus define a sheaf $\s{L}_{P^{gp}}$ of integral linear functions on $B$.  We similarly define a sheaf $\s{PL}_{P^{gp}}$ of integral piecewise linear functions.

We note that to specify an integral linear structure on an integral piecewise linear manifold (i.e., a manifold where transition functions are piecewise linear), it suffices to identify which $\bb{R}$-valued piecewise linear functions are actually linear.  These functions can then be used to construct charts.  It therefore also suffices (in dimension $2$) to specify which piecewise-straight lines are straight, since (piecewise-)straight lines form the fibers of (piecewise-)linear functions.

\subsection{Constructing $U^{\trop}$} \label{Ut}

Fix a toric model $(Y,D) \rar (\?{Y},\?{D})$, and let $N$ be the cocharacter lattice corresponding to $(\?{Y},\?{D})$.  Let $\Sigma\subset N_{\bb{R}}$ be the corresponding fan.  $\Sigma$ has cyclically ordered rays $\rho_i$, $i=1,\ldots,n$, with primitive generators $v_i$, corresponding to boundary divisors $\?{D_i}\subset \?{D}$ and $D_i \subset D$.  We choose an orientation\footnote{Choosing a cyclic ordering for the components of $D$ (assuming $D$ has at least three components) is equivalent to choosing an orientation for $N_{\bb{R}}$ or $U^{\trop}$.  It is also equivalent to fixing the sign for the holomorphic volume form $\Omega$ on $U$, which we will use in \S \ref{period}.  We assume throughout the paper that such a choice has been fixed.} of $N_{\bb{R}}$ so that $\rho_{i+1}$ is counterclockwise of $\rho_i$.  Let $\sigma_{u,v}$ denote the closed cone bounded by two vectors $u,v$, with $u$ being the clockwise-most boundary ray.  In particular, if $u$ and $v$ lie on the same ray, we define $\sigma_{u,v}$ to be just that ray.  Denote $\sigma_{i,i+1}:=\sigma_{v_i,v_{i+1}}$.  We may use variations of this notation, such as $v_{\rho}$ for a primitive generator of some arbitrary ray $\rho$ with rational slope, but these variations should be clear from context.

We now use $(Y,D)$ to define an integral linear manifold $U^{\trop}$.  As an integral piecewise-linear manifold, $U^{\trop}$ is the same as $N_{\bb{R}}$, with $0$ being a singular point and $U^{\trop}(\bb{Z}):=N$ being the integral points.   
 Note that an integral $\Sigma$-piecewise linear (i.e., bending only on rays of $\Sigma$) function  $\varphi$ on $U^{\trop}$ can be identified with a Weil divisor 
 of $Y$ via $W_{\varphi}:=a_1D_1+\ldots + a_n D_n$, where $a_i = \varphi(v_i)\in \bb{Z}$.  We define the integer linear structure of $U^{\trop}$ by saying that a function $\varphi$ on the interior of $\sigma_{i-1,i}\cup\sigma_{i,i+1}$\footnote{We assume here that there are more than $3$ rays in $\Sigma$, so that $\sigma_{i-1,i}\cup \sigma_{i,i+1}$ is not all of $N_{\bb{R}}$.  This assumption can always be achieved by taking toric blowups of $(Y,D)$.  Alternatively, it is easy to avoid this assumption, but the notation and exposition becomes more complicated.  We will therefore continue to implicitely assume that there are enough rays for whatever we are trying to do, without further comment.} is linear if it is $\Sigma$-piecewise linear and $W_\varphi \cdot D_{i} = 0$.  This last condition is (for $n\geq 2$) equivalent to
\begin{align} \label{linear}
	a_{i-1} +D_{i}^2 a_{i} + a_{i+1} = 0.
\end{align}

\begin{rmk}\label{charts}
This construction of $U^{\trop}$ naturally generalizes to higher dimensions, but the two-dimensional case is special in that the linear structure on $U^{\trop}$ is canonically determined by $(Y,D)$; i.e., it does not depend on the choice of toric model.    This is evident from the following atlas for $U^{\trop}$ (from \cite{GHK1}): the chart on $\sigma_{i-1,i}\cup \sigma_{i,i+1}$ takes $v_{i-1}$ to $(1,0)$, $v_i$ to $(0,1)$, and $v_{i+1}$ to $(-1,-D_i^2)$, and is linear in between.

Furthermore, toric blowups and blowdowns do not affect the integral linear structure, so as the notation suggests, $U^{\trop}$ and $U^{\trop}(\bb{Z})$ depend only on the interior $U$.
\end{rmk}

\begin{eg}
If $(Y,D)$ is toric, then $U^{\trop}$ is just $N_{\bb{R}}$ with its usual integral linear structure.  This follows from the standard fact from toric geometry that 
$
	\sum_i (C \cdot D_i) v_i = 0
$ 
for any curve class $C$.  Taking non-toric blowups changes the intersection numbers, resulting in a non-trivial monodromy about the origin.
\end{eg}

\begin{rmk}\label{divisorial}
 Recall from standard toric geometry that any primitive vector $v\in N$ corresponds to a prime divisor $D_v$ supported on the boundary of some toric blowup of $(\?{Y},\?{D})$, and a general vector $kv$ with $k\in \bb{Z}_{\geq 0}$ and $v$ primitive corresponds to the divisor $kD_{v}$.  Two divisors on different toric blowups are identified if they determine the same discrete valuation on the function field of $\?{Y}$ (equivalently, if there is some common toric blowup on which their proper transforms are the same).  Since taking proper transforms under the toric model gives a bijection between boundary components of $(Y,D)$ and boundary components of $(\?{Y},\?{D})$ (and similarly for the boundary components of toric blowups), we see that points of $U_0^{\trop}(\bb{Z})$ correspond to the divisorial discrete valuations of $(Y,D)$ along which a certain form $\Omega$ has a pole.  Here, $\Omega$ is the canonical (up to scaling) holomorphic volume form on $U$ with a simple pole along $D$, and {\it divisorial} means the valuation corresponds to a divisor on some toric blowup of $(Y,D)$.  $0\in U^{\trop}(\bb{Z})$ of course corresponds to the trivial valuation.
\end{rmk}

\begin{eg} \label{cubicdevelop}
Consider the cubic surface $(Y,D=D_1+D_2+D_3)$ constructed by taking two non-toric blowups on each of the three boundary divisors $\?{D_1}$, $\?{D_2}$, and $\?{D_3}$ of $\?{Y}:=\bb{P}^2$.  So we have $D_i^2=-1$ for each $i$.  Consider the universal cover $\xi:\wt{U}^{\trop}\rar U^{\trop}_0$ with the pulled-back integral linear structure.  Let $v_1^0$ be an element in the preimage of $v_1\in U^{\trop}_0$.  Designate $v_1^0$ as the clockwise-most ray of the ``$0^{\tt{th}}$'' sheet of $\wt{U}^{\trop}$, and let $v_i^j$ denote the element of $\xi^{-1}(v_i)$ on the $i^{\tt{th}}$ sheet.  Similarly for the rays $\rho_i^j$ they generate.  Let $\delta:\wt{U}^{\trop}\rar \bb{R}^2$ denote the linear map which takes $v_1^0$ to $(1,0)$ and $v_2^0$ to $(0,1)$.  Then one finds that Equation \ref{linear} forces $\delta(v_3^0)=(-1,1)$ and $\delta(v_i^j) = (-1)^j \delta(v_i^0)$.  See Figure \ref{cubic dev fig}.  We thus see that the monodromy of $U^{\trop}$ in this case is $-\id$.

\begin{figure}
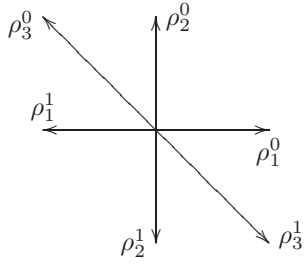

\[
\xy
{\ar (0,0); (15,0)}; (15,-3)*{\rho_1^0};
{\ar (0,0); (0,15)}; (3,15)*{\rho_2^0};
{\ar (0,0); (-15,15)}; (-18,14)*{\rho_3^0};
{\ar (0,0); (-15,0)}; (-15,3)*{\rho_1^1};
{\ar (0,0); (0,-15)}; (-3,-15)*{\rho_2^1};
{\ar (0,0); (15,-15)}; (18,-14)*{\rho_3^1};
\endxy
\]
\caption{Cubic surface developing map.}\label{cubic dev fig}
\end{figure} 
\end{eg}

\subsection{Convex Integral piecewise linear Functions on $U^{\trop}$}

 If we choose a monoid $P$ in our lattice $P^{gp}$, we can define what it means for a $\s{PL}_{P^{gp}}$ function $f$ to be convex along some ray $\rho$.  Let $\sigma^+$ and $\sigma^-$ denote disjoint open convex cones in $U^{\trop}$ with $\rho$ contained in each of their boundaries.  Let $n_{\rho}$ be the unique primitive element of $\Lambda^*$ which vanishes along the tangent space to $\rho$ and is positive on vectors pointing from $\rho$ into $\sigma^+$.  We note that $n_{\rho}$ may be viewed as $\pm v_\rho \wedge \cdot$, with the sign being positive if $\sigma_+$ is chosen to be counterclockwise of $\rho$.
 
 Observe than any integral linear function $f$ can be given on a cone $\sigma$ by some $f_{\sigma} \in \Lambda^*$, using the local embedding of $\sigma$ in its tangent spaces.  Since the cotangent spaces on either side of $\rho$ can be identified via parallel transport, we can compute
\begin{align*}
	 f_{\sigma^+}-f_{\sigma^-} = p_{\rho,f}n_{\rho}.
\end{align*}
Here, $p_{\rho,f}\in P^{gp}$ is called the {\it bending parameter} of $f$ along $\rho$.  Note that this is independent of which side of $\rho$ we call $\sigma^+$ and which we call $\sigma^-$.  We say that $f$ is {\it convex} (resp. strictly convex) along $\rho$ if $p_{\rho,f} \in P$ (resp. $P\setminus P^{\times}$, where $P^{\times}$ denotes the invertible elements of $P$).  We note that these notions naturally generalize to all integral linear manifolds.

For the rest of this section we will assume $P^{gp}=\bb{Z}$ and $P=\bb{Z}_{\leq 0}$.

\subsubsection{Piecewise Linear Functions in terms of Weil Divisors} \label{PLF}

Let $\varphi$ be a rational piecewise linear function on $U^{\trop}$ (that is, we are allowing rational values at integral points).  We will always assume that we have taken enough toric blowups of $(Y,D)$ so that $D_v \subset D$ for every $\rho_v$ along which $\varphi$ bends.  As in \S \ref{Ut}, we define a rational Weil divisor 
\begin{align*}
	W_{\varphi} := \sum_i \varphi(v_i) D_i.
\end{align*}
 Then it follows from Equation \ref{linear} that $p_i:= W_{\varphi} \cdot D_i$ is the bending parameter of $\varphi$ along $\rho_i$. 

Conversely, for any nonsingular compactification $(Y,D)$ of $U$ and any rational Weil divisor $W=\sum_i w_i D_i$ supported on $D$, there is a unique rational piecewise linear function $\varphi_W$ taking values $w_i$ on $v_i$ and bending only on the $\rho_i$'s.  $\varphi_W$ is integral if and only if $W$ is integral.  The bending parameter at $\rho_i$ is given by $W\cdot D_i$.  That is, if we view $W$ as a vector $W=(w_1,\ldots,w_n)$ in $\langle D \rangle$ (the lattice freely generated by the $D_i$'s), then the bending parameters of $\varphi_W$ are given by the vector 
\begin{align*}
	P=(p_1,\ldots,p_n)=H W,
\end{align*}
where $H =\left( D_i \cdot D_j\right)$ is the intersection matrix.  So given a collection of bending parameters $p_i$, there is a unique rational piecewise linear function on $U^{\trop}$ with these bending parameters if and only if $H$ is invertible, and it is given by the $\bb{Q}$-Weil divisor $W=H^{-1}P$.

Assume for now that $H$ is invertible over $\bb{Q}$.  Let $v\in U^{\trop}_0(\bb{Z})$.  We have $v=p_v v'$ for some non-negative integer $p_v$ and some primitive vector $v'$ on the ray $\rho_v$.  Let $\beta_{v}$ denote the unique rational piecewise linear function on $U^{\trop}$ which bends only on $\rho_v$ with bending parameter $-p_v$.  Note that the sums of functions of this form are exactly the convex rational piecewise linear functions on $U^{\trop}$ with integral bending parameters.

Let $\psi_{\rho_v}$ denote the unique convex {\it integral} piecewise linear function which bends only on $\rho_v$ with the smallest (in absolute value) possible nonzero bending parameter $b_v$ ($b_v$ may have to be less than $-1$ to ensure that $\psi_{\rho_v}$ can be integral).  The following proposition illustrates the utility of this Weil divisor perspective for understanding functions on $U^{\trop}$.

\begin{prop} \label{BetaSymmetry}
Assume $H$ is invertible over $\bb{Q}$.  For $v,w \in U^{\trop}(\bb{Z})$, we have $\beta_{v}(w) = \beta_{w}(v)$, and $\psi_{\rho_v}(b_w w') = \psi_{\rho_w}(b_v v')$
\end{prop}

\begin{proof}
Fix a compactification $(Y,D=D_1+\ldots+D_n)$, and view $D_v=p_v D_{v'}$ and $D_w=p_w D_{w'}$ as vectors in $\langle D_1,\ldots,D_n \rangle$.  Then $W_{\beta_v} = H^{-1} D_v$, and we have 
\begin{align*}
 \beta_{v}(w) = D_w^{T} H^{-1} D_v.
\end{align*}
So the first part of the proposition follows from the fact that the intersection form is symmetric.  The second part then follows because $\psi_{\rho_v} = \beta_{b_v v'}$.
\end{proof}

\subsection{Lines and Polygons in $U^{\trop}$} \label{valu}

Understanding lines and polygons in $U^{\trop}$ is important when studying compactifications of the mirror.  This will be essential when we investigate the tropicalizations of the theta functions in \S \ref{trop}.

\subsubsection{Lines in $U^{\trop}$} \label{lines}

By a ``line'' in $U^{\trop}$, we will mean a geodesic with respect to the canonical flat connection on $U^{\trop}_0$.  That is:
\begin{dfns}
A {\it parametrized line} in $U^{\trop}$ is a continuous map $L:\bb{R}\rar U^{\trop}_0$ 
 such that $L'(t_1)$ and $L'(t_2)$ are related by parallel transport along the image of $L$ for all $t_1,t_2\in \bb{R}$.   A {\it line} is the data of the image $L(\bb{R})$ and the vectors $L'(t)\in TU^{\trop}_0$, $t\in\bb{R}$, for some parametrized line $L$ (equivalently, a line is a parametrized line up to a choice of shift $t\mapsto t+c$ of the domain).  We may abuse notation by letting $L$ denote the unparametrized line or its image.

The {\it (signed) lattice distance} of a (parametrized) line from the origin is defined to be
\begin{align*}
	\dist(L,0): = L(t) \wedge L'(t)
\end{align*}
where $t$ is any point in $\bb{R}$, and the point $L(t)$ is identified with a vector in its tangent space.  Note that $d(L,0)>0$ means $L$ is going counterclockwise about the origin.

Now, for $q \in U^{\trop}_0$ and $d \in \bb{R}$, we define $L_q^d$ to be the line which goes to infinity parallel to $q$ and has lattice distance $d$ from the origin.  By {\it going to infinity parallel to} $q$ we mean that for any open cone $\sigma \ni q$, there is some $t_\sigma\in \bb{R}$ such that $t>t_\sigma$ implies $L(t) \in \sigma$ and $L'(t)=q$ under parallel transport in $\sigma$.  

We may similarly define {\it coming from infinity parallel to} $q$ by replacing $t>t_\sigma$ with $t<t_\sigma$ and replacing $L'(t)=q$ with $-L'(t)=q$.  We denote the directions in which a line $L$ goes to and comes from infinity by $L(\infty)$ and $L(-\infty)$, respectively.

\end{dfns}

\begin{rmk}
In general, a line need not go to or come from infinity at all.  In fact, one characterization of $U$ being positive is that every line in $U^{\trop}$ both goes to and comes from infinity.
\end{rmk}

\begin{dfn}\label{ZeroLine}
We define $L_q^0$ to be the limit of $L_q^{d}$ as $d$ approaches $0$ from below.  In other words, it consists of the ray coming in from the direction $L_q^{d<0}(-\infty)$ and hitting $0$, as well as the ray leaving the origin in the direction $q=L_q^{d<0}(\infty)$.  When we use the term ``line,'' we will be excluding the $d=0$ cases unless $L_q^0$ is invariant under the monodromy.
\end{dfn}

We say that a line $L_q$ {\it wraps} if it intersects every ray, except possibly $\rho_q$, at least once.  It wraps $k$ times if it hits each ray at at least $k$ times, except possibly for $\rho_q$, which it might only hit $k-1$ times.

We call the connected component of $U^{\trop}\setminus L$ containing the origin the $0$-side of $L$, denoted $Z(L)$.  We say a line $L_q^d$ has $0$ on the left if $d>0$, and on the right if $d<0$.  We will write $L_q^{d>0}$ or $L_q^{d<0}$ when we want to clarify that $0$ is on the left or right side, respectively, without having to specify $d$.  
  Let $L_q^{d,0} \subseteq L_q^d$ denote the boundary of the $0$-side.  Note that  $L_q^{d,0} = L_q^d$ exactly when the line does not self-intersect.

\begin{egs}
\hspace{.01 in}
\begin{itemize}[noitemsep]
\item If $(Y,D)$ is toric, then $U^{\trop} \cong \bb{R}^2$, and lines are just the usual notion of lines with a chosen constant velocity.
\item If $(Y,D)$ is the cubic surface from Example \ref{cubicdevelop}, then for any ray $\rho \subset {U^{\trop}}$, $U^{\trop}\setminus \rho$ is isomorphic (as an integral linear manifold) to an open half-plane.  Any line will go to and come from infinity in the same direction---we call such lines {\it self-parallel}.  If we now make a non-toric blowup on some $D_{\rho_q}$, then in the new integral linear manifold, $L_{q'}^{d}$ ($d\neq 0$) will self-intersect if $q'\neq q$, but will still be self-parallel if $q'=q$.  We will see in \S \ref{trop} that $L_{q'}^{d<0}$ self-intersecting corresponds to the theta function $\vartheta_{q'}$ having poles along every boundary divisor.
\item See Figure \ref{linefig} for illustrations of some possible lines.  
\end{itemize}
\end{egs}

\subsubsection{Polygons in $U^{\trop}$}

\begin{dfns} 
\hspace{.01 in}
\begin{itemize}[noitemsep] 
\item A {\it polytope} $\Delta \subset U^{\trop}$ is the closure of a set homeomorphic to an open $k$-ball for some $k\leq 2$ such that the boundary is a finite union of line segments and rays.  We also consider a point to be a polytope.  
 By {\it polygon}, we will mean a $2$-dimensional polytope.
 \item A polytope $\Delta$ is {\it convex} if any line segment in $U^{\trop}$ (including those which wrap around the origin) with endpoints $\Delta$ is entirely contained in $\Delta$.  
\item A polytope is {\it integral} (resp. {\it rational}) if all of its vertices are integral (resp. rational) points.  
\item A polygon is {\it nonsingular} if at each vertex of the form $v=F_1\cap F_2$ ($F_i$ edges), we have that primitive generators of $F_1$ and $F_2$ generate the lattice $\Lambda_p$ of integral tangent vectors at $p$.
\end{itemize}
\end{dfns}

We will be especially interested in polygons with $0$ in their interiors.

\begin{lemdfn}\label{PolyLemma} \begin{itemize}[noitemsep] Suppose that lines in $U^{\trop}$ all go to and come from infinity (equivalently, $U$ is positive).  Also, let $P^{gp}=\bb{Z}$, and $P=\bb{Z}_{\leq 0}$.    We then have:
\item A star-shaped (i.e., closed under multiplication by elements of $[0,1]$) {\it polygon} is a set $\Delta_{\varphi}\subset U^{\trop}$ of the form $\varphi \geq -1$ for some piecewise linear function $\varphi$ on $U^{\trop}$.
\item $\Delta_{\varphi}$ is {\it convex} if and only if $\varphi$ is convex.  Equivalently, the star-shaped polygon is convex if it is the closure of the intersection of a finite number of $0$-sides of lines in $U^{\trop}$, or equivalently, if it is convex on some cone-neighborhood of each vertex in the usual sense.  
\item $\Delta_{\varphi}$ is bounded if and only if $\varphi<0$ everywhere on $U_0^{\trop}$.  
\end{itemize}
\end{lemdfn}

\section{Construction of Theta Functions and the Mirror} \label{mirror}

This section summarizes \cite{GHK1}'s construction of the mirror family. We assume from now on that $(Y,D)$ is positive, unless otherwise stated.  This assumption simplifies the details of the construction, the notation, and the statements of the theorems from \cite{GHK1}, but the basic ideas of the construction are unchanged.  In \S \ref{compact}, we describe how to obtain fiberwise compactifications of the mirror as in \cite{GHK2}.  These compactifications do require positivity.

\subsection{Setup} \label{setup}

Choose some lattice $P^{gp}$ and some convex rational polyhedral cone $P_{\bb{R}}\subset P^{gp}_{\bb{R}}:=P^{gp}\otimes \bb{R}$.  Define $P:=P_{\bb{R}}\cap P^{gp}$.  When constructing the mirror over $\Spec \kk[P]$, we will need a choice $\varphi$ of ``multi-valued'' convex integral $\Sigma$-piecewise linear function.  As in \S \ref{int lin fun}, $\s{L}_{P^{gp}}$ (resp., $\s{PL}_{P^{gp}}$) denotes the sheaf of $P^{gp}$-valued integral linear functions (resp., integral piecewise linear functions) on $U^{\trop}$.  
  The sheaf of {\it multi-valued} integral piecewise linear functions is defined to be the quotient sheaf $\s{PL}_{P^{gp}}/\s{L}_{P^{gp}}$.  Note that the equivalence class of such a function is uniquely determined by choosing its bending parameters.

\begin{egs}
\hspace{.01 in}
\begin{itemize}
 \item One may take $P^{gp}:=A_1(Y,\bb{Z}) \cong \Pic(Y)^{*}$ and $P$ to be the Mori Cone $\NE(Y)$.  The fact that $P_{\bb{R}}:=\NE(Y)_{\bb{R}}$ is finitely generated here follows from the Cone Theorem and our assumption that $(Y,D)$ is positive.\footnote{When working without the positivity assumption, \cite{GHK1} chooses some strictly convex rational polyhedral cone containing $\NE(Y)_{\bb{R}}$ to be $P_{\bb{R}}$ ($\sigma_P$ in their notation).}  Take $\varphi:=\varphi_{\NE(Y)}$ to be the multi-valued integral piecewise linear function whose bending parameter along $\rho_i$ is $[D_i]$ for each $i$.  We may refer to the resulting mirror as the ``universal mirror family.'' 
 \item   Taking another choice of $P^{gp}$ and $P$ together with a surjective monoid homomorphism $\eta:\NE(Y)\rar P$ will define a subfamily.  For example, since $(Y,D)$ is positive, there exists a strictly effective $D$-ample Weil divisor $W$ supported on $D$ (strictly effective meaning effective with support equal to $D$). 
  Define $\eta_W:A_1(Y,\bb{Z})\rar \bb{Z}=:P^{gp}$, $C\mapsto -W\cdot C$.  Then the multi-valued function $\eta_W\circ \varphi_{\NE(Y)}$ has bending parameters $-W\cdot D_i$ along $\rho_i$, and thus is represented by the single-valued function $-\varphi_W$ as in \S \ref{PLF}.  In particular, $-\varphi_W$ is strictly convex for $P:=Z_{\leq 0}$.   The resulting family is a $1$-parameter subfamily of the universal one.
 \item\label{GSlocus} Let $E_1,\ldots,E_s$ denote the exceptional divisors of some toric model $\pi$ for $(Y,D)$.  Let $\NE_{\pi}(Y)$ denote the subcone of $A_1(Y,\bb{Z})$ spanned by $\NE(Y)$ and $-[E_i]$, $i=1,\ldots,s$.  We can then take $\eta$ to be the inclusion $\eta_{\pi}:\NE(Y)\hookrightarrow \NE_{\pi}(Y)$.  The resulting open subfamily of the universal family is used in \S 3 of \cite{GHK1} to relate their scattering diagrams to those of \cite{GPS}.
 \end{itemize} 
\end{egs}

We will always assume that $\varphi$ is given by $\eta\circ \varphi_{\NE(Y)}$ as in the examples above.

\cite{GHK1} defines a certain $P_{\bb{R}}^{gp}$-principal bundle $r:\bb{P}_{\varphi}\rar U^{\trop}$ which we may view $\varphi$ as a section of.\footnote{For $\varphi=\eta_W\circ \varphi_{\NE(Y)}$ as in the second example above, we can take $\bb{P}_{\varphi}$ to be the trivial bundle with $\varphi=\varphi_W$.  Since we are really only interested in tropicalizations of theta functions in this paper, restricting to this situation would be sufficient, and one can in this way avoid worrying about multi-valued functions.}  $\bb{P}_{\varphi}$ is defined as follows: Let $U_i$ denote $\sigma_{v_{i-1},v_{i+1}}$.  Choose representatives $\varphi_{U_i}$ of $\varphi$ on $U_i$ for each $i$, and glue a local trivialization $U_1\times P_{\bb{R}}^{gp}$ to a local trivialization $U_2\times P_{\bb{R}}^{gp}$ by identifying $(x,p) \in (U_1\cap U_2)\times P_{\bb{R}}^{gp} \subset U_1 \times P_{\bb{R}}^{gp}$ with $(x,p+\varphi_{U_2}(x)-\varphi_{U_1}(x)) \in (U_1\cap U_2)\times P_{\bb{R}}^{gp} \subset U_1 \times P_{\bb{R}}^{gp}$.  Since $\varphi_{U_2}-\varphi_{U_1}$ is an integral linear function, we see that $\bb{P}_{\varphi}$ has an integral linear structure.  When viewing $\varphi$ as a section of $\bb{P}_{\varphi}$, we may write $\wt{\varphi}$ to avoid confusion.

\subsubsection{The Cone Bounded by $\varphi$}\label{conephi}

As mentioned above, $\bb{P}:=\bb{P}_{\varphi}$ has an integral linear structure.  The fiber over $0$ is the singular locus, and the integral points are $\bb{P}_{\bb{Z}}:=\wt{\varphi}(U^{\trop}(\bb{Z}))+P^{gp}$.  Define $\tau_{\bb{R}}:=\wt{\varphi}(U^{\trop})+P_{\bb{R}}$, and let $\tau:=\tau_{\bb{R}}\cap \bb{P}_{\bb{Z}}$.  Note that the flat connection on $U^{\trop}_0$ lifts to a flat connection on $r_0:\bb{P}_0\rar U^{\trop}_0$ (the subscript $0$ will always mean we are taking the complement of the $0$-fiber).

Recall that $\Sigma$ denotes the fan with a ray $\rho_i$ for each $D_i \subset D$, 
and $\sigma_{i,i+1}$ is the closed cone bounded by $\rho_i$ and $\rho_{i+1}$.  We have a cones $\tau_{i,i+1,\bb{R}}:=\tau_{\bb{R}} \cap r^{-1}(\sigma_{i,i+1})$ with integral points $\tau_{i,i+1}:=\tau \cap r^{-1}(\sigma_{i,i+1})$.

Now define $\s{P}_{\bb{R}}:=(r_0)_{*} T\bb{P}_0$, the pushforeward of the tangent bundle of $\bb{P}_0$, and the subset $\s{P}:=(r_0)_{*} \Lambda\bb{P}_0$, the pushforeward of the integral tangent vectors.  Note that for any point $x$ in a set $S\subset U^{\trop}$ which is contained in the complement of some ray, there is an identification of $r^{-1}(S)$ with a subset of $(\s{P}_{\bb{R}})_x$, unique up to a linear function and identifying the points in $r^{-1}(S)\cap \bb{P}_{\bb{Z}}$ with points in $\s{P}_x$.  Furthermore, this identification can be chosen to commute with parallel transport along any path contained in $S$.  We may therefore write $\s{P}_S$ to mean $\s{P}_x$ for arbitrary $x \in S$.  For example, we have embeddings of $\tau_{i,i+1,\bb{R}}$ and $\tau_{i,i+1}$ into $\s{P}_{\sigma_{i,i+1}}$.  We will similarly denote $\s{P}_{S,\bb{R}}:=(\s{P}_{\bb{R}})_x$ for $x\in S$.  We will use these identifications freely.

\subsubsection{The Toric Case}\label{toric case}

If $(Y,D)$ is a toric variety with its toric boundary, then $\bb{P}$ can be identified with the vector space $U^{\trop}\times P^{gp}_{\bb{R}}$ (using \cite{GHK1}, Lemma 1.14), and this induces a monoid structure on $\tau$ (usually the monodromy about the $0$-fiber prevents us from having this global monoid structure).  In this case, the mirror family $\s{V}$ is simply $\Spec(\kk[\tau]) \rar \Spec(\kk[P])$, where the morphism comes from the inclusion of $P$ into $r^{-1}(0)$.   This is the well-known Mumford degeneration.  The central fiber is $\bb{V}_n:= \bb{A}_{x_1,x_2}^2 \cup  \bb{A}_{x_2,x_3}^2 \cup \ldots \cup \bb{A}_{x_n,x_1}^2 \subset \bb{A}_{x_1,\ldots,x_n}^n$ ($n \geq 3$), and the general fiber is $(\kk^*)^2$ (cf. \cite{GHK1}, \S 1.2).

Also in the toric case, given a convex integral polygon $\Delta$ in $U^{\trop}$, we can define a convex integral polygon $\Delta^{\?{\varphi}}:= \?{\varphi}(\Delta)+P_{\bb{R}} \subset \bb{P}$. 
 The corresponding toric variety $\s{V}_{\Delta^{\?{\varphi}}}$ is then a (partial) compactification of $\s{V}$.

In a non-toric case we do not have a natural global way to add points of $\tau$.  However, the identification with a cone in the tangent space does give us a natural monoid structure on $r^{-1}(\sigma)$ for any convex cone $\sigma$ in $U^{\trop}$.  
 Consider $\tau_{\rho_i}:=\tau_{i-1,i}+\tau_{i,i+1} \subset \s{P}_{\rho_i}$. 
 Now for any $\rho \subseteq \sigma\subset U^{\trop}$ ($\rho$ and $\sigma$ cones of dimension $1$ or $2$), define 
\begin{align}
    \tau_{\rho,\sigma}:=\tau_{\rho}-\wt{\varphi}(\sigma\cap U^{\trop}(\bb{Z})) = \{ x-y\in \s{P}_{\rho,\bb{R}} | x\in \tau_{\rho}, y \in \wt{\varphi}(\sigma\cap U^{\trop}(\bb{Z}))\}.
\end{align}
That is, we allow negation of integral points on the image of $\wt{\varphi}|_{\sigma}$.  Define $R_{\rho,\sigma}:=\kk[\tau_{\rho,\sigma}]$, and $\s{V}_{\rho,\sigma}:=\Spec(R_{\rho,\sigma})$.  Note that $R_{\rho,\sigma}$ is the localization of $R_{\rho,\rho}$ by functions of the form $z^{\wt{\varphi}(x)}$ for $x\in \sigma\cap U^{\trop}(\bb{Z})$.

The plan for constructing the mirror family is then to glue $\s{V}_{\rho_i,\rho_i}$ to $\s{V}_{\rho_{i+1},\rho_{i+1}}$ for each $i$, via an isomorphism $R_{\rho_i,\sigma_{i,i+1}} \wt{\rar} R_{\rho_{i+1},\sigma_{i,i+1}}$.  We do naturally have $R_{\rho_i,\sigma_{i,i+1}}$ identified with $R_{\rho_{i+1},\sigma_{i,i+1}}$ by parallel transport in $\sigma_{i-1,i}\cup \sigma_{i,i+1}$, but this naive identification is not the correct gluing: it gives a flat deformation of $\bb{V}_n^0:=\bb{V}_n\setminus \{0\}$, but this does not extend to a deformation of $\bb{V}_n$ (except in the toric case).  The problem is essentially that locally defined functions generally do not commute with transportation around the origin.  We therefore need a modified version of this gluing.

The correct modifications are defined in terms of a certain canonical {\it scattering diagram} in $U^{\trop}$.  We will also need an automorphism of $R_{\rho_i,\rho_i}$ for each $i$, and we will think of these as isomorphisms between $R_{\rho_i,\rho_i}^+:=R_{\rho_i,\rho_i}$ (thought of as corresponding to the cone $\sigma_{i,i+1}$) and $R_{\rho_i,\rho_i}^-:=R_{\rho_i,\rho_i}$ (associated with the cone $\sigma_{i-1,i}$).  Plus signs and minus signs as superscripts will always have these meanings for us.

\subsection{The Consistent Scattering Diagram} \label{scat}
A scattering diagram $\f{d}$ includes the data of a set of rays in $U^{\trop}$ with associated functions which satisfy certain conditions.  These functions are used to define certain ring automorphisms, and for the ``consistent" scattering diagram which we will define, these automorphisms make it possible to construct the scheme we were after in the previous subsection.

For a ray $\rho \subset U^{\trop}$ with rational slope, let $D_{\rho}$ be the corresponding boundary divisor in $(\wt{Y},\wt{D})$ (some toric blowup $\pi$ of $(Y,D)$).  Let $\beta \in H_2(\wt{Y},\bb{Z})$ with $k_{\beta}:=\beta\cdot D_{\rho} \in \bb{Z}$, and $\beta\cdot D_{\rho'}=0$ for $\rho\neq \rho'$.  Let $F_{\rho}:= \?{D\setminus D_{\rho}}$, $\wt{Y}_{\rho}^{\circ} := \wt{Y}\setminus F_{\rho}$, and $D_{\rho}^0:=D\setminus F_{\rho}$.

Now, define $\?{\s{M}}(\wt{Y}^{\circ}_\rho/D^{\circ}_{\rho},\beta)$ to be the moduli space of stable relative maps\footnote{For details on relative Gromov-Witten invariants, see \cite{Li}, or see \cite{GPS} for a treatment of this particular situation.} of genus $0$ curves to $\wt{Y}^{\circ}_{\rho}$, representing the class $\beta$ and intersecting $D_{\rho}^{\circ}$ at one unspecified point with multiplicity $k_{\beta}$.  This moduli space has a virtual fundamental class with virtual dimension $0$.  Furthermore, $\?{\s{M}}(\wt{Y}^{\circ}_\rho/D^{\circ}_{\rho},\beta)$ is proper\footnote{See Theorem 4.2 of \cite{GPS}, or Lemma 3.2 of \cite{GHK1}.} over $\Spec \kk$.  Thus, we can define the relative Gromov-Witten invariant $N_{\beta}$ as
\begin{align*}
    N_\beta := \int_{[\?{\s{M}}(\wt{Y}_\rho/D_{\rho},\beta)]^{vir}} 1.
\end{align*}
This is a virtual count of the number of curves in $\wt{Y}$ of class $\beta$ which intersect $D$ at precisely one point on $D_{\rho}^{\circ}$.  If $N_\beta \neq 0$, we call $\beta$ an $\bb{A}^1$ class.

Recall that $\eta$ denotes a homomorphism from $\NE(Y)$ to $P$.  We now define
\begin{align*}
    f_{\rho}:= \exp \left[ \sum_\beta k_{\beta} N_{\beta} z^{\eta(\pi_{*}(\beta))-\wt{\varphi}(k_{\beta}v_{\rho})} \right] \in R_{\rho,\rho}.
\end{align*}
Here, the sum is over all $\beta \in \NE(\wt{Y})$ which have $0$ intersection with all boundary divisors except for $D_{\rho}$. 

\begin{eg} \label{ScatterExample}
Let $\?{Y}=\bb{P}^2$ and let $\?{D}=\?{D_1}+\?{D_2}+\?{D_3}$ be a triangle of generic lines in $\?{Y}$.  Consider the pair $(Y,D)$ obtained by preforming a single non-toric blowup at a point on $\?{D_1}$.  Let $\beta=E_1$ be the exceptional divisor.  Then $N_\beta = 1$.  Due to the stacky nature of $\?{\s{M}}(Y_\rho/D_{\rho},\beta)$, $N_{\beta}$ might not always be a positive integer.  For example, with $Y$ and $\beta$ as above, we have $N_{k\beta} = \frac{(-1)^{k-1}}{k^2}$ (see \cite{GPS}, Proposition 6.1).

These multiple covers of $E_1$ are the only $\bb{A}^1$ classes for $D_1$, so we can compute $f_{\rho_1}$.  Suppose $P^{gp} := A_1(Y)$ and $\eta:=\id$.  
  We have 
\begin{align*}
    f_{\rho_1} &= \exp \left[ \sum_{k\in \bb{Z}_{> 0}} k\left( \frac{(-1)^{k-1}}{k^2} \right)z^{k[E_1]-\varphi(kv_{\rho_1})-kv_{\rho_1}} \right] \\
               &= 1+z^{[E_1]-\varphi(v_{\rho_1})-v_{\rho_1}}.
\end{align*}

More generally, if the only $\bb{A}^1$-classes hitting $D_{\rho}$ are a set $\{E_1,\ldots,E_k\}$ of $(-1)$-curves, along with their multiple covers, then 
\begin{align*}
    f_{\rho} = \prod_{i=1}^k \left(1+z^{\eta(E_i) -\wt{\varphi}(v_{\rho})}\right)  
\end{align*}
\end{eg}

\subsection{Constructing the Mirror Family}

The family $\s{V}$ we wish to construct will be a flat affine deformation of $\bb{V}_n$, but we will first construct a flat formal deformation $\hat{\s{V}}$ of $\bb{V}_n$.  This of course comes from an inverse system of infinitesimal deformations ${\s{V}_k}$ of $\bb{V}_n$.

  Note that $P\setminus 0$ corresponds to a maximal ideal $\f{m}\subset \kk[P]$.  Thus, for any $\kk[P]$-algebra $R$ and any $k\in \bb{Z}_{\geq 0}$, we have an ideal $\f{m}^k R$.

As explained in \S \ref{toric case}, we want to use the scattering diagram to glue $\s{V}_{\rho_i,\rho_i}^+$ to $\s{V}_{\rho_{i+1},\rho_{i+1}}^-$ by identifying $\s{V}_{\rho_i,\sigma_{i,i+1}}$ with $\s{V}_{\rho_{i+1},\sigma_{i,i+1}}$. Since the scattering diagram generally has infinitely many rays, we cannot usually do this directly.

Instead, we note that there are only finitely many rays $\rho$ in the interior of $\sigma_{i,i+1}$ for which the function $f_{\rho}\not\equiv 1$ modulo $\f{m}^k R_{\rho,\rho} \subset \f{m}^k R_{\rho_i,\sigma_{i,i+1}} = \f{m}^k R_{\rho_{i+1},\sigma_{i,i+1}}$. This is because there are only finitely many points in $P\setminus k\f{m}_P$, and $\bb{A}^1$-classes with non-vanishing contributions live in $\NE(Y)\setminus k\f{m}_{\NE(Y)}$.  We therefore replace each ring $R$ of the construction with $R_k:=R/\f{m}^k R$.

Now, given a curve $\gamma:[0,1]\rar U^{\trop}_0$, we will define a corresponding homomorphism $\Pi_{\gamma}^{(\pm,\pm)}:\kk[\s{P}^{\pm}_{\gamma(0),k}] \rar \kk[\s{P}^{\pm}_{\gamma(1),k}]$.  The signs in the superscripts are explained below, and the $k$'s in the subscripts indicate that we are modding out by $\f{m}^k$.  This homomorphism comes from using parallel transport of $\s{P}$ along $\gamma$, except whenever $\gamma$ crosses a scattering ray $\rho$ with $f_\rho \not\equiv 1$ modulo $\f{m}^k R_{\rho,\rho}$, we apply the $\kk[\s{P}_{\rho,k}]$-automorphism  
\begin{align}\label{wallcross}
    z^u\mapsto z^u f_{\rho}^{\langle n_{\rho}, r_{*}(u)\rangle},
\end{align}
where $n_{\rho}$ is a primitive generator of $\Lambda_{\rho}^*$ which is $0$ along $\rho$ and positive on vectors pointing into the cone from which $\gamma$ came, and $\langle \cdot,\cdot\rangle$ denotes the dual pairing. 
 Of course, if $\gamma(0)$ and/or $\gamma(1)$ are contained in scattering rays, we need to specify whether or not we apply the automorphisms corresponding to these rays.  If the first sign of the superscript of $\Pi_{\gamma}^{(\pm,\pm)}$ is $+$ (resp. $-$), the decision of whether or not to begin with the scattering automorphism corresponding to $\gamma(0)$ is determined by viewing $\gamma(0)$ as lying infinitesimally counterclockwise (resp. clockwise) of the ray it sits on, and similarly for $\gamma(1)$ with the second sign.

    Now, we can identify $\s{V}^+_{\rho_i,\sigma_{i,i+1},k}:=\Spec(R^+_{\rho_i,\sigma_{i,i+1},k})$ with $\s{V}^-_{\rho_{i+1},\sigma_{i,i+1},k}$ using the $\kk[\s{P}_{\sigma_{i,i+1}}]$-automorphism given by $\Pi_{\gamma}^{+,-}$, where $\gamma(0) \in \rho_i$, $\gamma(1) \in \rho_{i+1}$, and $\gamma \subset \sigma_{i,i+1}$.  We thus glue $\s{V}_{\rho_i,\rho_i,k}^+$ to $\s{V}_{\rho_{i+1},\rho_{i+1},k}^-$ for all $i$.  
  Similarly, for each $i$, we can glue $\s{V}^-_{\rho_{i},\rho_i,k}$ to $\s{V}^+_{\rho_{i},\rho_i,k}$ via the automorphism $\Pi_{\gamma}^{-,+}$, where $\gamma(t)=v_i\in \rho_i$ for all $t\in [0,1]$.

Preforming all these gluings yields schemes $\s{V}_{k}^{\circ}$ which are flat infinitesimal deformations of $\bb{V}_n^{\circ}:=\bb{V}_n\setminus \{0\}$ over $\Spec (\kk[P]/\f{m}^k)$.  Then $\s{V}_k:=\Spec \Gamma(\s{V}_{k}^{\circ},\s{O}_{\s{V}_{k}^{\circ}})$ is a flat infinitesimal deformation of $\bb{V}_n$ over the same base.  Taking the inverse limit of the deformations $\s{V}_k\rar \Spec (\kk[P]/\f{m}^k)$ with respect to $k$ yields a flat formal deformation $\hat{\s{V}}\rar \Spf \wh{\kk[P]}$ of $\bb{V}_n$ over the formal spectrum of the $\f{m}$-adic completion of $\kk[P]$.  Finally, in these positive cases we can take the affinization \[\s{V}:=\Spec\Gamma(\hat{\s{V}},\s{O}_{\hat{\s{V}}})\] over $\Spec \kk[P]$.  This is the space on which we shall focus.

\subsection{Broken Lines and the Canonical Theta Functions} \label{theta}
In this section we describe a canonical $\kk[P]$-module basis for the global sections of $\s{O}_{\s{V}}$.  These sections are called {\it theta functions}.  The mod-$\f{m}^k$ version of this construction is used in \cite{GHK1} to show that the spaces $\s{V}_k$ and $\hat{\s{V}}$ indeed have relative dimension $2$ over the base (as opposed to, say, $\s{O}_{\s{V}_{k}^{\circ}}$ having only $\kk[P]/\f{m}^k$ as its global sections, which might be the case if we did not use a consistent scattering diagram).

\begin{dfns}
Let $q \in U^{\trop}(\bb{Z})$, and $Q\in U_0^{\trop}$.  A {\it broken line} $\gamma$ with limits $(q,Q)$ is the data of a continuous map $\gamma:(-\infty,0]\rar U^{\trop}$, values $-\infty<t_0<t_1<\ldots<t_s=0$, and for each $t\neq t_i$, $i=0,\ldots,s$, an associated monomial $c_t z^{m_{t}} \in R_{\gamma(t)}:=\kk[\s{P}_{\gamma(t)}]$ with $c_t \in \kk$ and $r_{*}(m_{t})=-\gamma'(t)$, such that: 
\begin{itemize}[noitemsep]
\item $\gamma(0)=Q$
\item $\gamma_0:= \gamma|_{(-\infty,t_0]}$ and $\gamma_i:=\gamma|_{[t_{i-1},t_i]}$ are geodesics (i.e., straight lines with constant velocities).
\item For all $t\ll t_0$, $\gamma(t)$ is in some fixed convex cone $\sigma_q$ containing $q$, and $m_t=\wt{\varphi}(q)$ under parallel transport in $\sigma_q$.
\item For all $a \in (t_{i-1},t_i)$ (or $(-\infty,t_0)$ for $i=0$) and $b \in (t_i,t_{i+1})$, and all relevant $R_{\gamma(t)}$'s identified using parallel transport along $\gamma$, we have that $\gamma(t_i)$ is contained in a scattering ray $\rho$, and 
\begin{align*}
    c_b z^{m_{b}} = (c_{a} z^{m_{a}})(c_{\rho}z^{m_{\rho}})
\end{align*}
where $c_{\rho}z^{m_{\rho}}$ is any term in the formal power series expansion of $f_{\rho}^{\langle n_{\rho}, r_{*}(m_a)\rangle}$ (so $c_bz^{m_b}$ is a monomial term from the expansion of Equation \ref{wallcross}).
\end{itemize}

\begin{rmk}
We call the choice of monomial $c_{\rho} z^{m_{\rho}}$ a {\it bend}.  Note that broken lines in this setup can only bend away from the origin.  If we say that a bend is maximal, we will mean that the broken line is bending away from the origin as much as possible (that is, the degree of  $z^{-v_{\rho}}$ in the chosen monomial was as large as possible, so in particular $f_{\rho}$ must have been a polynomial).  We may also call this the maximal bend away from the origin.  In \S \ref{SeedScatter} we will see a related scattering diagram in $U^{\trop}$ equipped with a different linear structure.  In this situation, some broken lines may bend towards the origin, and we will be interested in the broken lines with the maximal allowed bends towards the origin (which in our current setup are always straight lines).
\end{rmk}

We say that two broken $\gamma$ and $\gamma'$ with $\Limits(\gamma)=(q,Q)$ and $\Limits(\gamma')=(q,Q')$ are {\it equivalent} if they have the same bends (so there is a natural correspondence between the smooth segments of the broken lines, with corresponding segments being parallel).  Let $[q,\gamma]$ denote the equivalence class of a broken line $\gamma$ with limits $(q,Q)$ (the inclusion of $q$ in the notation here is meant to simplify notation in the formulas below).

We say that an equivalence class $[q,\gamma]$ is {\it infinitely near} a ray $\rho$ ($[q,\gamma] \IN \rho$ for short)  if given any open cone $\sigma$ containing $\rho$, there exits a broken line $\gamma'\in [q,\gamma]$ with limits $(q,Q')$ such that $Q' \in \sigma$.  We say $[q,\gamma]$ is {\it positively infinitely near} $\rho$ ($[q,\gamma] \PIN \rho$) if the same is true for any half-open cone $\sigma^+$ containing $\rho$ as a clockwise-most boundary ray.  Similarly for {\it negatively infinitely near} ($[q,\gamma] \NIN \rho$) with $\sigma^-$ having $\rho$ a counterclockwise-most boundary ray.

Now at last we define the {\it theta functions}.  Given a class $[q,\gamma]$, let $c_{\gamma} z^{m_{\gamma}}$ denote the monomial attached to the last straight segment of each $\gamma' \in [q,\gamma]$.  Define $\vartheta_0=1$.  For $q \in U^{\trop}_0(\bb{Z})$ and $\rho\subset U^{\trop}$, we define 
\begin{align*}
   \left.\vartheta_q\right|_{\s{V}_{\rho,\rho}^{+}} := \sum_{[q,\gamma] \PIN \rho} c_{\gamma} z^{m_{\gamma}}, 
   \hspace{.25 in}\mb{ and }\hspace{.25 in}
   \left.\vartheta_q\right|_{\s{V}_{\rho,\rho}^{-}} := \sum_{[q,\gamma] \NIN \rho} c_{\gamma} z^{m_{\gamma}}.
\end{align*}
 Since the $\s{V}_{\rho,\rho,k}^{\pm}$'s form an open cover of $\s{V}_k^{\circ}$ for each $k$, this suffices to define the restrictions of the theta functions to each $\s{V}_k$, hence to $\hat{\s{V}}$, and this determines the theta functions on $\s{V}$ as desired.
\end{dfns}

\begin{rmk}
The scattering diagram we use is called ``consistent'' because \cite{GHK1} shows that for any $q\in U^{\trop}(\bb{Z})$ and any curve $\gamma$ in $U^{\trop}_0$ with $\gamma(0)\in \rho_0$ and $\gamma(1) \in \rho_1$, we have (modulo any positive integer power of $\f{m}$)
\begin{align}\label{consistency}
    \Pi_{\gamma}^{(\pm_0,\pm_1)}\left(\vartheta_q|_{\s{V}_{\rho_0,\rho_0}^{\pm_0}}\right) =  \vartheta_q|_{\s{V}_{\rho_1,\rho_1}^{\pm_1}}.
\end{align}
That is, the sums of monomials determining the theta functions are ``parallel'' with respect to this modified parallel transport $\Pi$. 
 This is exactly what we need for the theta functions to be well-defined globally. 
\end{rmk}

\begin{thm}[{{\cite{GHK1}}}]\label{thetamult} The theta functions form a canonical $\kk[P]$-module basis for the space of global sections of $\s{V}$.  That is,
\begin{align*}
    \s{V}=\Spec\left(\bigoplus_{q\in U^{\trop}(\bb{Z})} \kk[P]\vartheta_q \right).
\end{align*}
Furthermore, the multiplication rule can be described as follows:  Given $q_1,\ldots,q_s,q \in U^{\trop}(\bb{Z})$, the $\vartheta_q$-coefficient of $\vartheta_{q_1}\cdots\vartheta_{q_s}$ is given by
\begin{align}\label{MultRule}
    \sum_{\substack{([q_1,\gamma_1],\ldots,[q_s,\gamma_s]) \\
                   [q_i,\gamma_i] \IN \rho_{q} \\
                    m_{\gamma_1}+\ldots+m_{\gamma_s} = q}}
                    c_{Q_1} \cdots c_{Q_s}.
\end{align}
\end{thm}

The part about the multiplication rule is easy to see after noting that $\vartheta_q$ is the only theta function with a $z^q$ term along $\rho_q$.

\begin{eg}\label{CubicThetas}
For the cubic surface case of Example \ref{cubicdevelop}, one easily sees that for any $q\in U^{\trop}(\bb{Z})$ and $s\geq 0$, the only broken lines which contribute to Equation \ref{MultRule} applied to $\vartheta_q^s$ are straight lines.  One can then see that for $q\neq 0$, $\vartheta_q^s=\sum_{k=0}^s c_k \vartheta_{kq}$, where $c_k$ is the number of $s$-tuples $(\epsilon_1,\ldots,\epsilon_s)$ such that each $\epsilon_i=\pm 1$ and $\sum \epsilon_i = k$.  Here, $(\epsilon_1,\ldots,\epsilon_s)$ corresponds to an $s$-tuple $[q_i,\gamma_i]$, $i=1,\ldots,s$ as in (\ref{MultRule}), where $Q_i$ is, say, counterclockwise of $\rho_q$ if $\epsilon_i=1$ and clockwise if $\epsilon_i=-1$.  For example, $\vartheta_q^2=2+\vartheta_{2q}$, and $\vartheta_q^3 = 3\vartheta_q + \vartheta_{3q}$.

For $A\in \SL_2(\bb{C})$, the same relationship holds between $[\Tr(A)]^s$ (corresponding to $\vartheta_q^s$) and $\Tr(A^k)$ (corresponding to $\vartheta_{kq}$) for $k=0,\ldots,s$ (typically expressed using certain Chebyshev polynomials).  For $q$ primitive, \cite{GHK2} has shown that $\vartheta_q$ agrees with certain traces of holonomies of local systems.  The above argument shows that this extends to non-primitive $q$.  It follows that the bases of \cite{GHK1} for the cubic surface example agree with the bases of \cite{FG0} for what they call $\s{X}_{\PGL_2(\bb{C}),\wh{S}}$ where $\wh{S}$ is the once-punctured torus or four-punctured sphere.
\end{eg}

\subsection{Another Construction of $U^{\trop}$}\label{SeedScatter}
We discuss here another point of view on the construction of $U^{\trop}$ that will be helpful to us later on.  Recall that each seed $S$ induces a linear structure on $U^{\trop}$.  $U^{\trop}$ with this linear structure may be identified with $N_{\bb{R}}=N\otimes \bb{R}$, where $(Y,D)\rar (\?{Y},\?{D})$ is the toric model corresponding to $S$ and $N$ is the cocharacter lattice of $\?{Y}$.  Suppose that this toric model includes $b_i$ non-toric blowups on $D_{m_i}$, with corresponding exceptional divisors $E_{ij}$, $j=1,\ldots,b_i$.

Now, let $\f{d}_0$ be the scattering diagram in $N_{\bb{R}}$ with rays 
\begin{align*}
\left\{\left.\bb{R} m_i, \prod_{j=1}^{b_i} \left(1+z^{\wt{\varphi}(m_i)-\eta(E_{ij})}\right)\right| i=1,\ldots,n\right\},
\end{align*}
where $\eta:=\eta_{\pi}$ is as in Example \ref{GSlocus}.  
 One may use $\f{d}_0$ to construct a consistent scattering diagram $S(\f{d}_0)$ as in \cite{KS} and \cite{GPS}.  All of the rays added to $\f{d}$ are outgoing, meaning that any broken line crossing these scattering rays can only bend away from the origin.  Thus, it is only broken lines crossing $\bb{R}_{\geq 0} m_i$ that can bend towards the origin.

$U^{\trop}$ with its usual integral linear structure now comes from modifying $N_{\bb{R}}$ so that lines which take the maximal allowed bend towards the origin are actually straight (cf. \S \ref{int lin fun}).  Furthermore, if we break our initial scattering rays up into two outgoing rays by negating the exponents of the $\bb{R}_{\geq 0} m_i$ parts of the initial rays, then $S(\f{d}_0)$ becomes our consistent scattering diagram $\f{d}$ in $U^{\trop}$ from before.  This construction is carried out in detail in \S 3 of \cite{GHK1}.

\subsection{The Cluster Complex}\label{clustercomplex}
We will now show that lines which do not wrap (cf. \S \ref{lines}) bound especially nice parts of the scattering diagram and correspond to particularly simple theta functions.  Recall that $\sigma_{u,v} \subset U^{\trop}$ denotes the cone with $u$ on the clockwise-most boundary ray and $v$ on the counterclockwise-most boundary ray.  Also recall our notation regarding lines in \S \ref{lines}.

\begin{lem}\label{model}
Let $q\in U^{\trop}(\bb{Z})$ and suppose $L_q^{d<0}$ does not wrap.  Let $q_{\pm}:=L_q^{d<0}(\pm\infty)\in U^{\trop}(\bb{Z})$ (so $q_+=q$).  There is some compactification $(Y,D)$ of $U$ which admits a toric model where all the non-toric blowdowns are on divisors $D_u$ with $u\in\sigma_{q_-,q_+}$ (cf. Figure \ref{linefig}(a), where we write $V_{-}$ and $V_+$ instead of $q_{-}$ and $q_+$).
\end{lem}
\begin{proof}
Let $v$ be any vector in the complement of $\sigma_{q_-,q_+}$ forming nonsingular cones with $q_+$ and $q_-$.  Let $(Y,D)$ have the form $D=D_{q_+}+D_v+D_{q_-}+\sum D_i$, where the $D_i$'s correspond to vectors in $\sigma_{q_-,q_+}$.  Note that $D_v^2=0$ because $q_+ + q_- = 0$ in $\sigma_{q_+,q_-}$.  Thus, $|D_v|$ gives a fibration $Y\rar \bb{P}^1$ with rational fibers and with $D_{q_+}$ and $D_{q_-}$ as sections.  Let $F$ be the fiber containing $\sum D_i$.  Since $Y$ is smooth, $Y\rar \bb{P}^1$ is obtained from a $\bb{P}^1$-bundle by a sequence of blowups, so the $\bb{P}^1$'s in $F$ not contained in $\sum D_i$ do not hit nodal points of $\sum D_i$.  These $\bb{P}^1$'s are then $(-1)$-curves (for $U$ generic in its deformation class) and can be blown down.  On the complement of $D_v$ and $F$, each fiber is a chain of $\bb{P}^{1}$'s.  We can contract all but one of these $\bb{P}^1$'s from each chain, and then what remains on the complement of $D$ is just a $\kk^*$ fibration over $\kk^*$; i.e., $(\kk^*)^2$.  Thus, we have constructed a toric model of the desired type.
\end{proof}

Note that this toric model is unique except for the choices of exceptional divisors intersecting $D_{q_-}$ and $D_{q_{+}}$.

\begin{cor}\label{FiniteScatter}
If $L_q^{d<0}$ does not wrap, then for $U$ generic, the only $\bb{A}^1$-classes corresponding to rays in $\sigma_{q_-,q_+}$ are exceptional divisors in one of these toric models.
\end{cor}
\begin{proof}
Suppose $C\subset (Y,D)$ is an $\bb{A}^1$ class for some $v\in \sigma_{q_-,q_+}$ such that $C$ is not contracted under one of these toric models.  Then in this toric model, $\?{C}\subset (\?{Y},\?{D})$ intersects only divisors corresponding to rays in one half of the plane $N_{\bb{R}}$.  Since $\sum (\?{C}\cdot \?{D}_v) v = 0$ for toric varieties, this is impossible unless $\?{C}$ only intersects $\?{D}_{q_-}$ and $\?{D}_{q_+}$.  In this case, $C$ is a component of a fiber other than $D_v$ and $F$ in the above proof, and such fibers are chains of $(\bb{P}^1)$'s.  Since $U$ is generic, we can assume the fiber contains only two $\bb{P}^1$'s, and either one can be contracted in a toric model for the proof of the previous lemma.
\end{proof}

\begin{dfn}
The {\it cluster complex} is the union of the cones of the form $\sigma_{q_-,q_+}$ as in Lemma \ref{model}. 
\end{dfn}

See \S \ref{ClusterRelation} for a brief explanation of how this relates to the usual notion of the cluster complex from \cite{FG1}.

Given this understanding of the cluster complex in the scattering diagram, we can describe many of the theta functions very explicitely.  Let $L_q^{d<0}$ and $\sigma_{q_-,q_+}$ be as above.  Note that for any $x\in \sigma_{q_-,q_+}$, the only broken lines with initial direction $q$ and endpoint $x$ must be going clockwise about the origin---they cannot wrap or go around the origin the other way because $\sigma_{q_+,q_-}$ is a half-space.  Thus, these broken lines will only hit the scattering rays in $\sigma_{q_-,q_+}$.  Let $\sigma_{q}$ be a top-dimensional non-singular cone with $q$ as the counterclockwise-most endpoint and containing no scattering rays in its interior (this is possible by Corollary \ref{FiniteScatter}).  Now for $x\in \sigma_q$, the only broken line with endpoint $x$ and initial direction $q$ is a straight broken line.  Thus,  on $\s{V}_{\sigma_q,\sigma_q}$, $\vartheta_q$ is given by the monomial $z^{\wt{\varphi}(q)}$.

Note that there exists an open immersion $\s{V}_{\sigma_q,\sigma_q}\hookrightarrow \s{V}$, given modulo $\f{m}^k$ by the restriction map \[\Gamma(\s{V}_k,\s{O}_{\s{V}_k})=\s{O}_{\s{V}_{k}^{\circ}}(\s{V}_{k}^{\circ}) \rar \s{O}_{\s{V}_{k}^{\circ}}(\s{V}^{\pm}_{\sigma_q,\sigma_q,k}).\]  Suppose we cross clockwise past a scattering ray in the interior of $\sigma_{q_-,q_+}$ to a cone corresponding to another patch of $\s{V}$.  Let $e$ be a primitive generator of the scattering ray $\rho_e$, and suppose that a toric model as in Lemma \ref{model} consists of $b_e$ blowups along $D_e$.  From Example \ref{ScatterExample}, we know that the scattering automorphism for crossing $\rho_e$ clockwise is given by 
\begin{align}\label{MonomialMutation}
    z^v\mapsto z^v \left(\prod_{i=1}^{b_e} \left(1+z^{\eta(E_i) -\varphi(e)-e}\right)\right)^{e\wedge r_*(v)}. 
\end{align}

In the language of cluster varieties, this corresponds to mutating the $\s{X}$-space with respect to the seed vectors which project to $\rho_e$.  The claim that $\vartheta_q=z^{\wt{\varphi}(q)}$ means that $\vartheta_q$ is a cluster monomial.  We thus demonstrate a special case of more general results relating $\s{V}$ to cluster varieties: \cite{GHK2} shows that $\s{V}$ can always (in any positive case) be identified with a space of deformations of $U$, which by Section 5 of \cite{GHK3} can essentially (up to codimension $2$ and after restricting to the big torus orbit of the base) be identified with a cluster $\s{X}$-variety.  \cite{GHKK} extends this relationship between $\s{V}$ and cluster varieties to higher dimensions, including the result relating theta functions to cluster monomials.

\subsection{Compactifications} \label{compact}
Let $\Delta$ be a convex rational nonsingular polytope in $U^{\trop}$ such that each vertex of $\Delta$ is contained in a ray of $\Sigma$.  Note that $\Sigma$ 
  induces a polyhedral decomposition $\Sigma\Delta$ on $\Delta$.   As in \cite{GHK2}, we construct from $\Sigma\Delta$ a partial (full if $\Delta$ is bounded) fiberwise compactification $\s{V}_{\Delta}$ of $\s{V}$.

First we recall that in the toric situation, the compactified family is the toric variety corresponding to the polytope $Q_{\Delta}:=\varphi(\Delta)+P_{\bb{R}}$.  The general fiber is the toric variety corresponding to $\Delta \subset N_{\bb{R}}$, while the central fiber is $\bb{V}_n(\Delta)$, a compactification of $\bb{V}_n$ where the irreducible components are the toric varieties corresponding to the cells of $\Sigma\Delta$ (cf. \cite{TorDeg}).

As in the construction of $\s{V}$, the idea behind the general construction is to do the toric construction locally on $U^{\trop}$ and to use the scattering diagram for gluing.  Given a maximal dimensional cell $\sigma\in \Sigma\Delta$, let $Q_{\sigma}$ denote the polytope $\varphi(\sigma)+P_{\bb{R}}$ embedded in $\s{P}_{\sigma,\bb{R}}$.  For any cell $\rho$ in $\Sigma\Delta$, define $Q_{\rho}=\bigcup_{\sigma\supset \rho} Q_{\sigma} \subset \s{P}_{\rho,\bb{R}}$, where the union is over the maximal dimensional cells containing $\rho$.  Now we define a cone $\kappa_{\rho,\bb{R}}\subseteq \s{P}_{\rho,\bb{R}}$ generated by
\begin{align*}
\{x-y \in \s{P}_{\rho,\bb{R}}: x\in Q_{\rho}, y \in \varphi(\rho)\}.
\end{align*}
Let $\kappa_{\rho}$ denote the integral points of $\kappa_{\rho,\bb{R}}$.  Note that if $\rho\in \Sigma$, then $\kappa_{\rho}$ is just $\tau_{\rho,\rho}$ from \S \ref{conephi}.

Thus, the new cones for this construction come from taking $\rho$ to be in a boundary component of $\Delta$.  If $F_{i,i+1}$ denotes the edge $\sigma_{i,i+1} \cap \partial {\Delta}$, and $p_i=F_{i-1,i}\cap F_{i,i+1}=\rho_i\cap \partial\Delta$, then $\kk[\kappa_{p_i}]$ is a subring of $\kk[\tau_{\rho_i,\rho_i}]$. $\Spec(\kk[\kappa_{p_i}])\setminus\Spec(\kk[\tau_{\rho_i,\rho_i}])$ contains two toric boundary divisors, corresponding to the faces sitting over $F_{i-1,i}$ and $F_{i,i+1}$.

Now, the construction of the compactified family $\s{V}_{\Delta}$ proceeds as for $\s{V}$, forming inverse systems of quotients of the $\kk[\kappa_{p_i}]$'s and using the scattering automorphisms to glue.  $\s{V}_{\Delta}\setminus \s{V}$ is a set of divisors $\{\s{D}_i\}$ corresponding to the $F_i$'s.

To show that this construction is well-defined and that each face really gives a single, well-defined boundary divisor, we have to check that $\s{D}_{F_{i,i+1}}:=\Spec \kk[\kappa_{F_{i,i+1}}] \setminus \Spec \kk[\kappa_{\sigma_{i,i+1}}]$ is preserved by scattering automorphisms corresponding to rays in $\sigma_{i,i+1}$.  Let $\rho_u$ be such a scattering ray, generated by primitive $u\in \sigma_{i,i+1}$.  Let $v$ be a primitive vector tangent to $F_{i,i+1}$.   Let $u'\in \Lambda$ be such that $u = ku' + lv$, $k,l \in \bb{Z}$, $k > 0$, and $u'$ and $v$ are a basis of $\Lambda$. Then $\kk[\kappa_{F_{i,i+1}}] = k[P][z^{\pm v},z^{-u'}]$, and $\s{D}_{F_{i,i+1}}$ is the zero set of $z^{-u'}$.  The automorphism for crossing $\rho_u$ takes a monomial $z^m$ to $z^m(1+z^{-u}f)$ for some power series $f$ in $z^{-u}$, so since $z^{-u}$ is zero along $\s{D}_{F_{i,i+1}}$, we see that valuations of rational functions along this divisor are unchanged by the scattering automorphism.  Thus, the boundary divisors do indeed locally look the same as in the toric situation where this is no scattering.  Note that this applies even when $\rho_u$ is in the boundary of $\sigma_{i,i+1}$, i.e., passing through a vertex of $\Delta$.

Let $L_{v_F}^{d_{F}>0}$ be the line containing some edge $F$ of $\Delta$.  Let $\rho$ be a ray intersecting $F$. 
 The valuation (i.e., the order of vanishing) of some $z^{(q,p)}\in R_{\rho_i,\rho_i}$ ($q=r((q,p))$) along the divisor $\s{D}_F$ is 
\begin{align}\label{val}
    \val_{\s{D}_F}\left(z^{(q,p)}\right)=v\wedge q,
\end{align}
where $v$ denotes the parallel transport of $v_F$ along $L_{v_F}^{d_{F}>0}$ to $F\cap \rho_i$.  We will use this to explicitely describe valuations of theta functions in the next section.

\section{Tropical Theta Functions} \label{trop}

\subsection{Tropicalization of the Mirror}\label{wU}
We know from \cite{GHK2} that generic fibers of the mirror $\s{V}$ are deformation equivalent to our the original space $U$.  Thus, the tropicalization $V^{\trop}$ of a generic fiber $V$ is non-canonically isomorphic to $U^{\trop}$, and any construction done using $U$ and $U^{\trop}$ can similarly be done using $V$ and $V^{\trop}$.  We describe here an explicit identification of $V^{\trop}$ with $U^{\trop}$.

\begin{ntn}
We will use gothic $\f{D}$'s to denote divisors on the boundary of a generic fiber $V$ of the mirror.  Script $\s{D}$'s  denote boundary divisors for the whole mirror family.  We will use $(Z,\f{D})$ to denote a compactification of $V$.
\end{ntn}

As we just saw in \S \ref{compact}, lines with rational slope in $U^{\trop}$ determine boundary divisors of $\s{V}$.  In the construction above, the divisor does not depend on the vector attached to the line or on the distance of the line from the origin.  Given a primitive vector $v\in U^{\trop}$, we can associate the divisor $\f{D}_{L_v^{d>0}}$ corresponding to $L_v^{d>0}$.  Similarly, for $v=|v|v'$ with $v'$ primitive and $|v|$ a non-negative rational number, we associate the divisor $|v|\f{D}_{L_v^{d>0}}$.  This gives an identification of $U^{\trop}(\bb{Q})$ with $V^{\trop}(\bb{Q})$ which restricts to an identification of $U^{\trop}(\bb{Z})$ with $V^{\trop}(\bb{Z})$.  We will see that this extends to an integral linear identification $w_U:U^{\trop}\rar V^{\trop}$.  This is the identification we will primarily use.

\begin{convention}
We give $V^{\trop}$ the opposite orientation of that induced by $w_U$.
\end{convention}

Alternatively, given $v=|v|v'$ as above, we can associate $|v|\f{D}_{L_v^{d<0}}$.  This is equivalent to doing the above identification with the orientation of $U^{\trop}$ reversed.  We will not use this identification $U^{\trop}\rar V^{\trop}$, but it is closely related to what we will call $w_V:V^{\trop} \rar U^{\trop}$ in \S \ref{sympair}.

\subsection{Tropicalizing Functions}
For any rational function $f$ on $V$, we define an integral piecewise linear function $f^{\trop}:V^{\trop}\rar \bb{R}$ as follows: for $v \in V^{\trop}(\bb{Z})$, $f^{\trop}(v) := \val_{\f{D}_v} (f)$.  We then extend $f^{\trop}$ linearly to the real points of $V^{\trop}$.

For this section, we once again call $\bb{R}$-valued functions convex if their bending parameters are non-positive (i.e., we take $P:=\bb{Z}_{\leq 0}$).

\begin{lem}\label{regconvex} 
If $f$ is regular on $V$, 
 then $f^{\trop}$ is convex.
\end{lem}
\begin{proof}
Let $(Z,\f{D})$ be a nonsingular compactification of $V$ such that any ray on which $f^{\trop}$ is nonlinear corresponds to some component of $\f{D}$.  The principal divisor corresponding to $f$ is $(f)=\f{D}^0_f-\f{D}^{\infty}_f+V(f)$, where $\f{D}^0_f$ denotes the divisor of zeroes of $f$ on the boundary, $\f{D}^\infty_f$ denotes the divisor of poles of $f$ on the boundary, and $V(f)$ denotes the interior zeroes of $f$.  So $f^{\trop}$ is the integral piecewise linear function on $V^{\trop}$ corresponding to the Weil divisor $\f{D}^{0}_f-\f{D}^{\infty}_f$, and the bending parameter along some $\rho_v$ is given by $\f{D}_v \cdot (\f{D}^0_f-\f{D}^{\infty}_f) = -\f{D}_v \cdot V(f) \leq 0$.
\end{proof}

The properties of valuations give us the following relations for all rational functions on $V$:
\begin{align*} 
  (fg)^{\trop} & = f^{\trop} + g^{\trop}
\end{align*}
\begin{align}
    (f+g)^{\trop} &\geq \min(f^{\trop},g^{\trop}) \label{valuations}
\end{align}
Furthermore, the second relation is an equality at points where $f^{\trop} \neq g^{\trop}$.  Suppose that there exists a $v \in {U^{\trop}}$ such that $(f+g)^{\trop}(v) > \min[f^{\trop}(v),g^{\trop}(v)]$.  Then, by continuity, there must be some open cone $\sigma$ in $U^{\trop}$ containing $v$ where $f^{\trop}=g^{\trop}$.  We will see that if $f$ and $g$ are theta functions, then having $f^{\trop}|_\sigma = g^{\trop}|_\sigma$ for non-empty open set $\sigma$ implies $f=g$.  So the inequality in Equation \ref{valuations} is an equality for theta functions, and similarly for any finite sum theta functions with positive coefficients.

\begin{rmk}\label{nocancel}
We will need that the monomials attached to the broken lines contributing to a theta function do not cancel with each other when added together.  This follows from a result in \cite{GHKK}, which shows that the monomials attached to broken lines all have positive coefficients.  In fact, this is already sufficient to show that Equation \ref{valuations} is an equality for theta functions.
\end{rmk}

\subsection{The Valuation Functions}\label{ValFun}

Given a vector $v\in U^{\trop}$, we define an integral piecewise linear function $\val_v:U^{\trop}\rar \bb{R}$ as follows.  For $d\leq 0$, the fiber $\{\val_v=d\}$ is the set $L_v^{-d,0}$ (so $0$ is on the left).  If $L_v^{-d}$ wraps, then this completely defines $\val_v$. 

If $L_v^{-d}$ does not wrap, then these fibers with $d<0$ miss some cone $\sigma\subset U^{\trop}$.  In this case, for $d>0$, the fiber $\{\val_v=d\}$ is the broken line with initial direction $v$ and signed lattice distance $-d$ from the origin which takes the maximal allowed bend across every scattering ray that it crosses.  By \S \ref{clustercomplex}, there are only finitely many such scattering rays.  We call this broken line $\f{L}_v^{-d}$.  Note that $\val_v$ is always convex---By construction, negative fibers can only bend towards the origin, while positive fibers only bend away from the origin, and this is equivalent to convexity.

Consider the cases where $L_v^{-d}$ does not wrap.  By taking a toric model corresponding to scattering rays in $\sigma$ as in Lemma \ref{model}, we can see that there is some seed $S$ with respect to which each $L_v^{-d>0}$ and each $\f{L}_v^{-d<0}$ is straight and goes to $\infty$ parallel to $v$.  With respect to this linear structure, $\val_v$ is given explicitely by $q\mapsto (v\wedge q)$.

Differentiating gives us a function $D\val_v:TU_{\val_v}^{\trop}\rar \bb{R}$, where $U_{\val_V}^{\trop}$ denotes the complement in $U^{\trop}_0$ of the singular locus of $D{\val_v}|_{U^{\trop}_0}$.  Note that if we identify $q$ with a vector $\wt{q}$ in its tangent space, then $D\val_v(\wt{q}) = \val_v(q)$.

\begin{lem}\label{DecreasingBroken}
Let $\gamma$ be a broken line with $m_{t}=-\gamma'(t)$ being ($r_*$ of) the attached monomial at some time $t$.  If $t_2 > t_1$, then $D\val_v(m_{t_2}) \geq D\val_v(m_{t_1})$ (assuming the $t_i$'s are generic enough for each side to be defined).
\end{lem}
As in \cite{GHKK}, we say that functions satisfying this condition for all broken lines are {\it decreasing along broken lines} (since they decrease on the tangent directions of the broken lines).
\begin{proof}
First note that $\val_v$ being convex means that the bends of $\val_v$ while moving along $\gamma$ will only increase $D\val_v(m_t)$, as desired.  Now let $\rho_u$ (the ray generated by some primitive $u$) be the only scattering ray where $\gamma$ bends between times $t_1$ and $t_2=t_1+\epsilon$. Then $m_{t_2}=m_{t_1}-ku$ for some $k\in \bb{Z}_{\geq 0}$.

  Suppose that $\val_v \leq 0$ everywhere.  In particular, $\val_v(u) \leq 0$.  Then 
\begin{align*}
D\val_v(m_{t_2}) &\geq D\val_v(m_{t_1}) - kD\val_v(u)  \\
                &= D\val_v(m_{t_1}) - k\val_v(u) \geq D\val_v(m_{t_1}).
\end{align*}

On the other hand, suppose $\val_v$ is positive somewhere.  Let $\sigma$ be the cone on which it is non-negative, and $S$ a corresponding seed as in Lemma \ref{model}.  Let $\gamma$ bend along some ray $\rho_u$ between times $t_1$ and $t_2=t_1+\epsilon$ as before. 
 If $u \notin \sigma$, then $\val_v(u) \leq 0$, and we again see $D\val_v(m_{t_2}) \geq D\val_v(m_{t_1})$.  Otherwise, we work with the linear structure and scattering diagram on $U^{\trop}$ corresponding to the seed $S$ (cf. \S \ref{SeedScatter}).  With respect to this structure, broken lines in $\sigma$ bend towards the origin, so $m_{t_2} = m_{t_1}+ku$, $k\in \bb{Z}_{\geq 0}$, and so we still have $D\val_v(m_{t_2}) \geq D\val_v(m_{t_1})$, as desired.
\end{proof}

Define $\Val_v(\vartheta_q):=\min_{[q,\gamma]}\left[\min_{t\in (-\infty,0]} D\val_v(-\gamma'(t))\right]$, where the first $\min$ is over all equivalence classes of broken lines with initial direction $q$.  More generally, for a function $f=\sum_{i\in I} a_i \vartheta_{q_i}$ with $a_i\neq 0$ for each $i\in I$, define $\Val_v(f) = \min_{i\in I} \Val_v(\vartheta_{q_i})$.  The above lemma implies:
\begin{cor}\label{valtheta}
$\Val_v(\vartheta_q)=\val_v(q)$.
\end{cor}

 \begin{lem}\label{AdditiveVal}
$\Val_v(\vartheta_{q_1}\vartheta_{q_2}) = \Val_v(\vartheta_{q_1}) + \Val_v(\vartheta_{q_2})$.  
 \end{lem}
 \begin{proof}
Suppose that $\val_v$ is non-positive everywhere.  Then it only bends along a single ray $\rho$.  If we take a branch cut along $\rho$, $U^{\trop}$ can be identified with a convex cone on which $\val_v$ is linear.  On the other hand, if $\val_v$ is positive somewhere then we have seen that there is some seed with respect to which $\val_v$ is linear.

In either case, Theorem \ref{thetamult} and Remark \ref{nocancel} imply that $\vartheta_{q_1} \vartheta_{q_2}$ has a $\vartheta_{q_1+q_2}$ term (addition performed with respect to the above-mentioned linear structure or branch cut on $U^{\trop}$ that makes $\val_v$ linear).  The linearity of $\val_v$ then gives us $\Val_v(\vartheta_{q_1+q_2}) = \Val_v(\vartheta_{q_1}) + \Val_v(\vartheta_{q_2})$.  On the other hand, Theorem \ref{thetamult} and Lemma \ref{DecreasingBroken} imply there is no term $\vartheta_q$ in the expansion $\vartheta_{q_1}\vartheta_{q_2}=\sum a_q\vartheta_q$ with $a_q\neq 0$ such that $\Val(\vartheta_q)<\Val_v(\vartheta_{q_1+q_2})$, so $\Val_v(\vartheta_{q_1}) + \Val_v(\vartheta_{q_2})$ indeed gives the desired minimum in the definition of $\Val_v(\vartheta_{q_1}\vartheta_{q_2})$.
 \end{proof}

\begin{thm}\label{ValFibers}
Under the identification $w_U:U^{\trop}(\bb{Z})\rar V^{\trop}(\bb{Z})$ described in \S \ref{wU}, $\val_v(q) = \val_{\f{D}_{v}}(\vartheta_q)$.  Thus, $\Val_v (f) = \val_{\f{D}_{v}}(f)$.
\end{thm}
\begin{proof}
Suppose that $q = L_v^{-d}(t_q)\in L_v^{-d,0}$ for some $d<0$.  We see from Equation \ref{val} and the definition of theta functions that
\begin{align}\label{thetat}
    \val_{\f{D}_{v}}(\vartheta_q)  = \min_{[q,\gamma] |\gamma(0)=q}  v\wedge m_{\gamma},
\end{align}
where $v$ may be interpreted as $\gamma'(t_q)$.  By the definition of $\Val_v$ and the fact that $v\wedge m_{\gamma} = D\val_v(m_{\gamma})$, the right-hand side is $\geq \Val_v(\vartheta_q)$, which by Corollary \ref{valtheta} equals $\val_v(q)$.  The straight broken line contained in $\rho_q$ has $m_{\gamma} = q$, and so this gives us equality.

Now suppose $\val_v(q) = d \geq 0$.  Let $p\in L_v^{c}$ for some $c>0$.  Then, as in Equation \ref{thetat}, we have 
\begin{align*}
    \val_{\f{D}_{v}}(\vartheta_q)  = \min_{[q,\gamma] |\gamma(0)=p}  v\wedge m_{\gamma},
 \end{align*}
    and this is still $\geq \Val_v(\vartheta_q)=\val_v(q)=d\geq 0.$

Now, pick any $q'$ with $\val_v(q')=d' < -d$.  We can write $\vartheta_q \vartheta_{q'} = \sum_{r\in I} a_r \vartheta_r$, $a_r \neq 0$, for some $I\subset U^{\trop}(\bb{Z})$.   Lemma \ref{AdditiveVal} tells us that $\Val_v(\vartheta_q\vartheta_{q'})=d-d'<0$.  In particular, there is some $r\in I$ with $\val_{\f{D}_{v}}(\vartheta_r) = \Val_v(\vartheta_r) =d-d'<0$, so we do not need to worry about the $r\in I$ for which $\val_{\f{D}_{v}}(\vartheta_r)\geq 0$.  The previous paragraph shows that these are the $r$ for which $\val_v(r) \geq 0$.

Thus, we have 
\begin{align*}
\val_{\f{D}_{v}}(\vartheta_q) + \val_{\f{D}_{v}}(\vartheta_{q'}) = \val_{\f{D}_{v}}(\vartheta_q \vartheta_{q'}) =\min_{r\in I} \val_{\f{D}_{v}} \vartheta_r 
                                                    =  \Val_v(\vartheta_q \vartheta_{q'}) = d-d'.
\end{align*}
So $\val_{\f{D}_{v}}(\vartheta_q)=d=\val_v(q)$, as desired.
\end{proof}

\subsection{Tropical Theta Functions} \label{ttf}

The previous subsection tells us that $\vartheta_q^{\trop}(v)=\Val_v(\vartheta_q) = \val_v(q)$.  In this subsection we will explicitely describe the fibers of $\vartheta_q^{\trop}$ in $V^{\trop}$.

\begin{ntn}
We will use the notation $\wedge_{q^+}$ to indicate we are using the wedge product defined on $U^{\trop}$ by cutting along $\rho_q$ and then identifying $\rho_q$ with the clockwise-most boundary ray of $U^{\trop}\setminus \rho_q$ (so $q\wedge v\geq 0$ for nearby $v$ in $U^{\trop}\setminus \rho_q$).  Similarly, for $\wedge_{q^-}$ we identify $\rho_q$ with the counterclockwise-most boundary ray.  Let $\mu$ denote the monodromy action on $TU_0^{\trop}$ corresponding to parallel transporting vectors counterclockwise about the origin.
\end{ntn}

\begin{lem}\label{tpos}
If $\val_v(q) < 0$, then 
\begin{align*}
\val_v(q) & =  \min_{\substack{
                                        t \in \bb{R}| L_v^{d> 0}(t)\in \rho_q}}   \left\{(L_v^{d>0})'(t)\wedge q \right\} \\
          & =  \min_{\substack{i=0,\ldots,k}}   \left\{ \mu^{-i}v \wedge_{v^+} q \right\}  \\
          & =  \min_{\substack{i=0,\ldots,k}}   \left\{v\wedge_{q_-} \mu^{i} q\right\}  \\
          & =  \min_{\substack{
                                        t\in \bb{R}| L_q^{d<0}(t)\in \rho_v}}   \left\{v\wedge (L_q^{d<0})'(t)\right\}  
\end{align*}
where $k$ is the smallest non-negative integer such that $v\wedge_{q_-} \mu^{k+1} q \geq 0$.
\end{lem}
\begin{proof}
Let $t_1,\ldots,t_k$ be the times at which $L_v^{d>0}(t)$ intersects $\rho_q$.  For the first equality, note that if for some $d_i$, $L_v^{d_i>0}(t_i)=q$, then $(L_v^{d_i>0})'(t)\wedge q$ is negative the lattice distance of the line from the origin at that time\footnote{When we multiply $d$ by a positive scalar $c$, we map $L_v^d(t)$ to $cL_v^d(t)$.  That way the times $t_1,\ldots,t_k$ are unchanged.} (i.e., $-d_i$).  Since $L_v^{d>0,0}$ contains the point of $\rho_q\cap L_v^{d>0}$ closest to the origin, say, $L_v^{d>0}(t_m)$, we have that $d_m$ is the largest of the $d_i$'s.  Hence, the $\min$ in the first equality is obtained at $L_v^{d>0}(t_m)\in L_v^{d>0,0}$.  Since $(L_v^{d>0})'(t_m)\wedge q = \val_v(q)$, this proves the first equality.

The second equality follows by noting that each time we follow $L_v^{d>0}$ around the origin (moving backwards along the line), the tangent vector (initially $v$) is multiplied by $\mu^{-1}$.  Note that $k$ as in the statement of the theorem is the number of times that the $L_v^{d>0}$ intersects $\rho_q$.

The third equality follows from the fact that $\mu\in SL_2(\bb{Z})$, and so $a\wedge b = \mu(a)\wedge \mu(b)$.  The fourth equality follows symmetrically to the second equality.  
\end{proof}

\begin{cor} \label{level}
Under the identification $w_U$ of $U^{\trop}$ with $V^{\trop}$, for $d<0$, $L_q^{d,0}$ (as defined in \S \ref{lines}) is the fiber $\{v\in U^{\trop}|\vartheta_q^{\trop}(v)=d\}$.  
\end{cor}

\begin{prop}\label{levelpos}
Under the identification $w_U$ of $U^{\trop}$ with $V^{\trop}$, for $d>0$, $\f{L}_q^{d}$ (as defined in \S \ref{ValFun}) is the fiber $\{v\in U^{\trop}|\vartheta_q^{\trop}(v)=d\}$. 
\end{prop}
\begin{proof}
Let $\gamma_q$ and $\gamma_v$ be broken lines with initial tangent vectors $q$ and $v$, respectively, which are supported on $\f{L}_q^{d,0}$ and $\f{L}_v^{-d,0}$, respectively.  Let $q_1,v_1$ be negative of the tangent vectors to $\gamma_q$ and $\gamma_v$, respectively, on the counterclockwise-side of a scattering ray $\rho_v$ generated by primitive vector $v$, and similarly for $q_2$ and $v_2$ on the clockwise-side of $\rho_u$.

In light of Theorem \ref{ValFibers}, it suffices to show that $v_1\wedge q_1 = v_2 \wedge q_2$.  Let $b_{u}$ be the degree of the scattering function attached to $\rho_u$ (so for $U$ generic, it is the number of $(-1)$-curves hitting $D_u$).  Then when crossing in the counterclockwise direction, $q_2$ changes to $q_1=q_2+b_{u}(u\wedge q_2) u$, while $v_2$ changes to $v_1=v_2 + b_u (u\wedge v_2) u$.  So indeed,
\begin{align*}
v_1 \wedge q_1 = v_2 \wedge q_2 + b_u (u\wedge v_2)(u\wedge q_2) + b_u(v_2\wedge u)(u\wedge q_2 ) = v_2 \wedge q_2.
\end{align*}
\end{proof}

\subsection{Symmetry of the Dual Pairing}\label{sympair}

Note that we have a canonical pairing $\langle \cdot,\cdot\rangle_{\bb{Z}}: U^{\trop}(\bb{Z})\times V^{\trop}(\bb{Z})\rar \bb{Z}$ defined by $\langle q,v\rangle := \vartheta_q^{\trop}(v)=\val_{\f{D}_v}(\vartheta_q)$.  This can be extended to a pairing $\langle \cdot,\cdot\rangle: U^{\trop}\times V^{\trop}\rar \bb{R}$ as follows:  extending to rational points is easy because the pairing is linear with respect to multiplication by non-negative rational (and real) numbers in either variable.  Fixing one variable gives a piecewise linear (in particular, continuous) function in the other, and so we can extend continuously to the real points for both variables.

On the other hand, since $V$ is itself a log Calabi-Yau surface (deformation equivalent to $U$), we could apply the mirror constructions of \S \ref{mirror} to a compactification $(Z,\f{D})$ of $V$ to construct a mirror family $\s{U}$, with points $v\in V^{\trop}(\bb{Z})$ corresponding to canonical theta functions $\vartheta_v$ on $\s{U}$.  Let $U'$ be a generic fiber of $\s{U}$.  We obtain a map $w_V:V^{\trop}\rar (U')^{\trop}$ analogously to how we defined $w_U$ (here, it is important to remember that we take the orientation of $V^{\trop}$ to be opposite that induced by $w_U$).  By composition we obtain an identification $w_V\circ w_U:U^{\trop} \rar (U')^{\trop}$, which is in fact induced by an identification of a deformation of $U$ with $U'$.  
 Corollary \ref{level} and Proposition \ref{levelpos} hold as before with the roles of $U^{\trop}$ and $V^{\trop}$ interchanged.  We see:

\begin{thm}\label{symdual}
For $q\in U^{\trop}$ and $v\in V^{\trop}$, $\vartheta_q^{\trop}(v) = \vartheta_v^{\trop}(q)$ under the identification $w_V\circ w_U$.  In other words, the pairing $\langle \cdot,\cdot \rangle$ does not depend on which side we view as the mirror.
\end{thm}
\begin{proof}
Note that the support of $w_U(L_q^{d,0})$ is the same as that of $L_{w_U(q)}^{-d,0}$, and similarly with $w_U(\f{L}_q^{d,0})$ and $\f{L}_{w_U(q)}^{-d,0}$.  The negation of the distance comes from the difference in orientation between $U^{\trop}$ and $V^{\trop}$.  We want to show that $\vartheta_q^{\trop}(v) = \val_q(v)$. This follows immediately from comparing the definition of $\val_q$ in \S \ref{ValFun} to the descriptions of $\vartheta_q^{\trop}$ in Corollary  \ref{level} and Proposition \ref{levelpos}.
\end{proof}

\subsection{Bending Parameters of Tropical Theta Functions}

\subsubsection{Bends of the Negative Fibers}\label{negbend}

\begin{figure} \begin{center}
    \begin{tabular}{ c  c c c}
    \includegraphics[scale=.3]{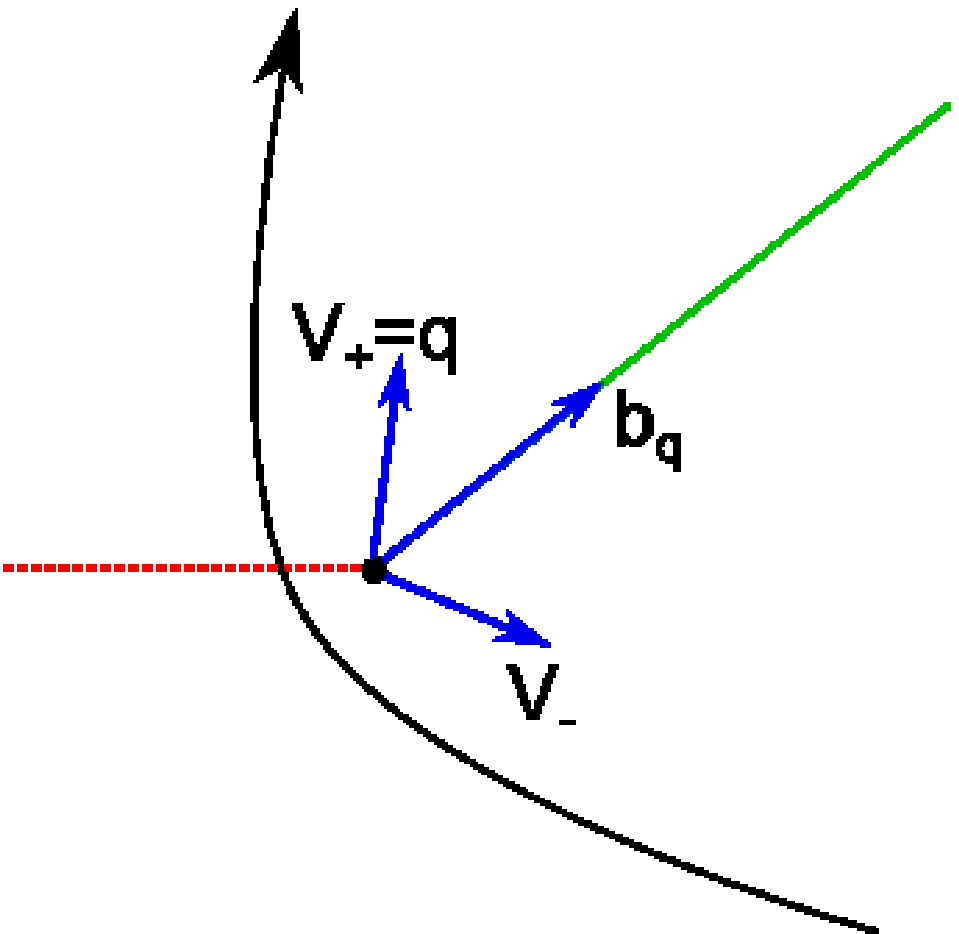} 
    &\includegraphics[scale=.3]{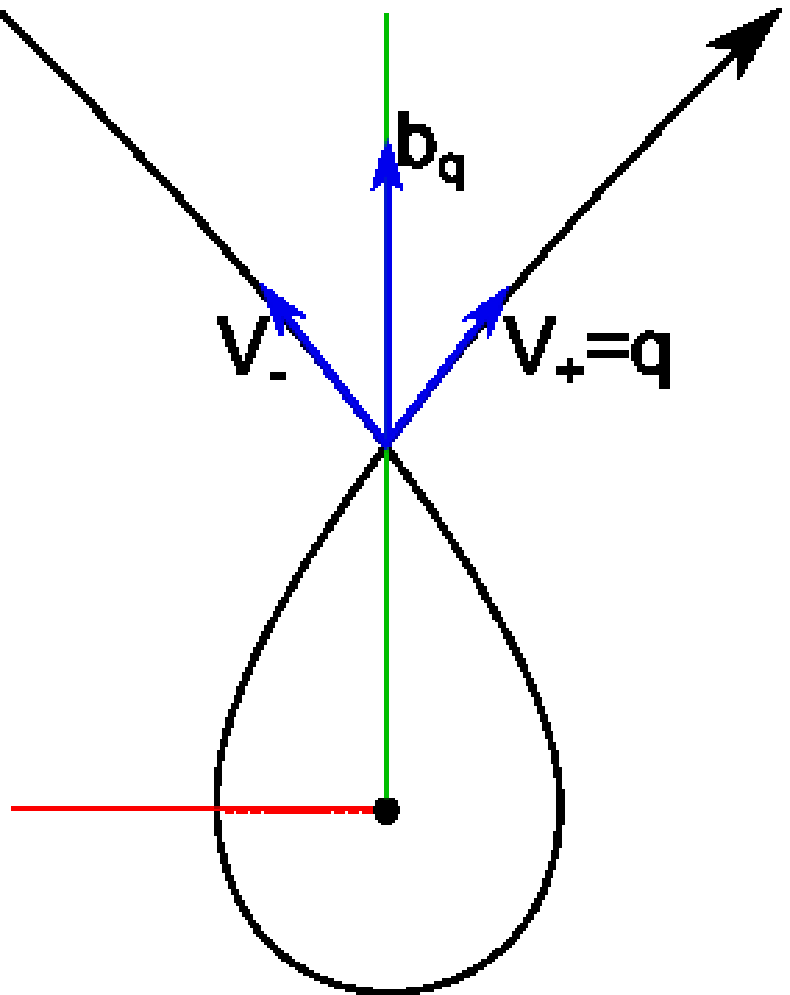} 
    &\includegraphics[scale=.38]{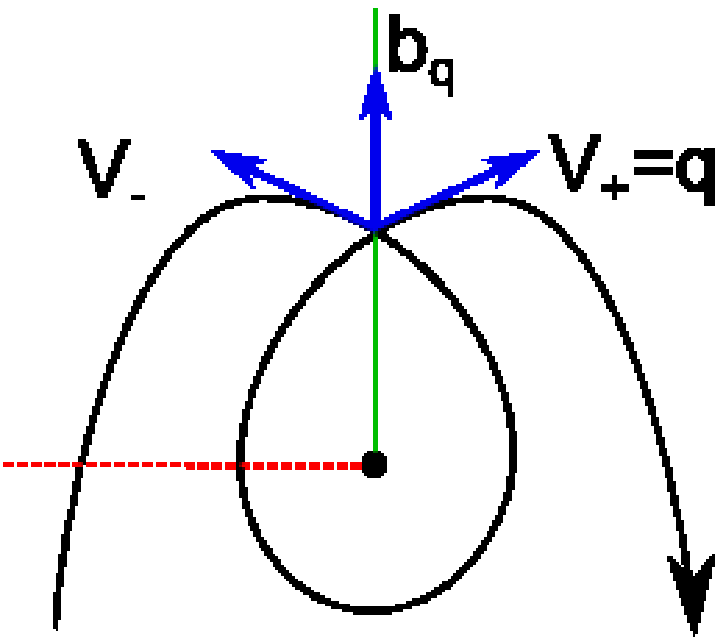} 
    &\includegraphics[scale=.38]{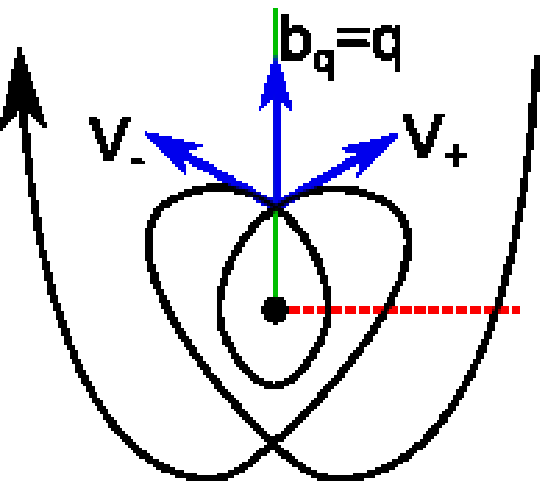} \\
        (a) & (b)  & (c)  & (d)
    \end{tabular}
\caption{Some lines $L_q$ which (a) do not wrap; (b) wrap once; (c) wrap twice (as in the $E_7$ case); and (d) wrap three times (as in the $E_8$ case).  The dashed red rays indicate our chosen branch cuts.  The blue vectors denote the boundary vectors and the bend $b_q$.  The green rays are the rays along which the functions bend.  The curved appearance of the lines occurs because the projection of $U^{\trop}$ onto the page is not an isometry. \label{linefig}}
\label{wrap}\end{center}
\end{figure}

Let $d<0$.  Recall that $L_q^{d,0}$ is the fiber $\vartheta_q^{\trop}(v)=d$ by Corollary \ref{level}.    Either $L_q^{d,0}$ is unbounded as in Figure \ref{linefig}(a), or, if $L_q^{d}$ self-intersects, then $L_q^{d,0}$ is bounded as in Figure \ref{linefig}(b,c,d).  It is clear from these figures that there is some $b_q\in V^{\trop}(\bb{Z})=w_U(U^{\trop}(\bb{Z}))$ such that the negative part of $\vartheta_q^{\trop}$ agrees with a function in $\beta_{b_q}$, where $\beta_{b_q}$ here is defined as in \S \ref{Betas} to be the set of functions with bending parameter $|b_q|$ along $\rho_{b_q}$.  Furthermore, the ray $\rho_{b_q}$ should intersect the vertex of $L_q^{d,0}$ (if there is one).  If $H =\left( D_i \cdot D_j\right)$ is invertible, then we saw in \S \ref{PLF} that each $\beta_{b_q}$ in fact consists of a single function, which we denote there as $\beta_{b_q}$ (abusing notation).

To find $b_q$, we first define the {\it boundary vectors} of $L_q^{d,0}$ (see Figure \ref{linefig}).  If $L_q^{d,0}$ is unbounded, then we say the boundary vectors of $L_q^{d,0}$ are $V_+:=L_q^d(\infty)=q$ and $V_{-}:=L_q^d(-\infty)$. Otherwise, let $t_1<t_2\in \bb{R}$ denote the initial and final times times for which $L_q^{d}(t)\in L_q^{d,0}$.  Then the boundary vectors are $V_{+}:=(L_q^d)'(t_2)$ and $V_{-}:=-(L_q^d)'(t_1)$ (so $V_+$ is the outward flow, and $V_{-}$ is the inward flow, which we negate).   Note that $V_-=-\mu(V_+)$.  We can add these tangent vectors and identify the sum with a point in $U^{\trop}$.
    We claim that ${b_q}:=V_-+V_+$.  In fact, this is easy to see: just observe that when we cross the ray $\rho_{b_q}$ in the counterclockwise direction, $\vartheta_q^{\trop}$ changes from $\cdot \wedge (-V_{-})$ to $\cdot \wedge V_+ = \cdot \wedge (-V_-+{b_q})$, which indeed means that the negative part of $\vartheta_q^{\trop}$ agrees with a function in $\beta_{b_q}$.

It follows immediately from the above argument that if $L_q^{d>0}$ wraps at most once (Figure \ref{linefig}(a,b)), then $b_q=q-\mu{q}$ (where we choose a cut which hits $L_q^d$ exactly once).  In terms of the classification of positive Looijenga interiors in \cite{Man2}, the $Q=E_7$ and $E_8$ cases\footnote{Consider $(\?{Y}=\bb{P}^2,\?{D}=\?{D}_1+\?{D}_2+\?{D}_3)$. The $E_7$ (resp. $E_8$) case refers to the Looijenga interiors $U=Y\setminus D$ which can be obtained by blowing up $2$ non-nodal points on $\?{D}_1$, $3$ on $\?{D}_2$, and $4$ (resp. $5$) on $\?{D}_3$.  $E_7$ and $E_8$ refer to the intersection forms on the lattices $D^{\perp}:=\{C\in \Pic(Y)|C \cdot D_i=0 \mbox{ for all } i\}$ in these cases.}  are the only ones where lines wrap more than once (Figure \ref{linefig}(c,d)).  We take cuts as in the figures.  In the $E_7$ case, we still have that ${b_q}=q-\mu(q)$.  In the $E_8$ case, we find ${b_q}=\mu(q)-\mu^2(q) = q$.

\subsubsection{Bends of the Positive Fibers}\label{positive} 

We continue to use the identification $w_U$.

\begin{prop}
Let $v\in U^{\trop}_0(\bb{Z})$ (identified with $V^{\trop}_0(\bb{Z})$ by $w_U$) be primitive, generating a ray $\rho$.  Suppose $\vartheta_q^{\trop}(v) \geq 0$, and assume that $\vartheta_q^{\trop}$ is positive somewhere.  Let $b_{\rho}$ be the degree of the scattering function $f_{\rho}$ (so for $U$ generic, it is the number of interior $(-1)$-curves intersecting $D_{\rho}$).  Then the bending parameter of $\vartheta_q^{\trop}$ along $\rho$ is $-b_{\rho}\vartheta_q^{\trop}(v)$.
\end{prop}
\begin{proof}
This follows immediately from Proposition \ref{levelpos}, the definition of broken lines, and the description of the scattering diagram in \S \ref{clustercomplex}.  Alternatively, we have the following more geometric proof.  Suppose $U$ is generic and $E_1,\ldots,E_{b_i}$ are the $(-1)$-curves intersecting $\f{D}_{\rho}$.  Let $\pi:Z\rar \?{Z}$ be a toric model for a compactification $Z$ of $V$ which contains these $E_i$'s as exceptional divisors, and let $z^m$ be any monomial on $\?{Z}$.  From Equation \ref{MonomialMutation} and the geometric description of mutations from \cite{GHK3}, the exceptional divisor $E_i$ comes from blowing up the intersection $\{1+z^{\eta(E_i') -\varphi(e)-e}=0\}\cap D_{\rho_v}$, where $E_i'$ is the corresponding exceptional divisor for $U$.  By Equation \ref{val}, the exponent in Equation \ref{MonomialMutation} applied to $z^m$ is $\val_{\f{D}_{\rho}}(z^m)$.  This implies that $\val_{E_i}(z^m)=\val_{\f{D}_{\rho}}(z^m)$, which implies $\val_{E_i}(\vartheta_q) = \vartheta_q^{\trop}(v)$ for each $i$.  Furthermore, all the zeroes of $\vartheta_q$ are along $(-1)$-curves like this.   The description of the bending parameters then follows from the relationship between bending parameters and intersection numbers in the proof of Lemma \ref{regconvex}.
\end{proof}

We can now prove:
\begin{prop}  \label{unique}
Suppose that two tropical theta functions $\vartheta_{q_1}^{\trop}$ and $\vartheta_{q_2}^{\trop}$ are equal on some open cone $\sigma \subseteq {U^{\trop}}$.  Then $q_1 = q_2$. 
\end{prop}
\begin{proof}
 Suppose that there is some subcone of $\sigma$ on which the functions are negative.  Then the fiber $\vartheta_{q_1}^{\trop}=\vartheta_{q_2}^{\trop}=-1$ is a line segment $L$ in $U^{\trop}$, and extending this segment to $\infty$ (with $0$ on the right) recovers $q_1=q_2$.

Now suppose that $\vartheta_{q_1}^{\trop}=\vartheta_{q_2}^{\trop} \geq 0$ everywhere on $\sigma$.  Recall that this means $\vartheta_{q_i}^{\trop}$ will bend along each $\rho_v\subset \sigma$ with bending parameter $-b_{\rho_{v_i}}\vartheta_{q_i}^{\trop}(v_i)$, where $v_i$ is primitive on $\rho_i$ and $b_{\rho_{v_i}}$ is as above.  Thus, we know how to extend the fibers to infinity, and this determines the $q_i$'s by Proposition \ref{levelpos}.
\end{proof}

\subsection{Convexity Properties}\label{Convexity}

We saw in Lemma \ref{regconvex} that tropicalizations of regular functions are convex.  \cite{GHKK} defines a stronger version of convexity, namely, {\it convexity along broken lines}.  Recall from \S \ref{int lin fun} that to define a linear structure on a piecewise linear manifold, it suffices to specify which piecewise-straight lines are straight.

\begin{dfn}
Let $\gamma:(-\infty,0] \rar U^{\trop}$ be a broken line, and let $\varphi$ be a rational piecewise linear function on $U^{\trop}$.  In a neighborhood of a point $\gamma(t_p)=p$ contained in a ray $\rho$, we can modify the linear structure of $U^{\trop}$ so that $\gamma'(t)$ is constant in a neighborhood of $t_p$ (with adjacent tangent spaces identified using parallel transport along $\gamma$). 
 Then $\varphi$ is said to be convex along $\gamma$ at the point $p$ if it is convex across $\rho$ with respect to this affine structure.  We say that $\varphi$ is {\it convex along broken lines} if it is convex along every broken line.
\end{dfn}
Note that the usual notion of convexity is just convexity along straight lines.  Our definition is somewhat different from that used in \cite{GHKK}.  They say a function is convex along broken lines if it is decreasing along broken lines, in the sense of \S \ref{ValFun}.  These definitions are in fact equivalent:

\begin{lem}\label{ConvDec}
Convex along broken lines is equivalent to decreasing along broken lines. 
\end{lem}
\begin{proof}
This follows from recalling that the usual notion of convexity can be defined as decreasing along straight lines.
\end{proof}

\begin{dfn}
We call a function $\varphi:U^{\trop} \rar \bb{R}$ {\it tropical} if it is integral piecewise linear and convex along broken lines.  Note that tropical functions are closed under addition and $\min$.  We say $\varphi$ is an indecomposable tropical function if it cannot be written as a minimum of some finite collection $S$ of tropical  functions with $\varphi \notin S$.
\end{dfn}

\cite{FG1} defines another notion of convexity: 
\begin{dfn}
Recall that every seed (or every toric model) induces a vector space structure on $U^{\trop}$.  One says that a piecewise linear function $\varphi:U^{\trop} \rar \bb{R}$ is {\it convex with respect to every seed} if it is convex with respect to each of these vector space structures.
\end{dfn}

Recall that we can apply the mirror construction to $V$ and $V^{\trop}$, so the notion of convexity along broken lines makes sense in $V^{\trop}$.  As we will see in the proof of Theorem \ref{IndeTropical} below, convexity along broken lines in $V^{\trop}$ is equivalent to convexity along straight lines and along maximally broken lines in the cluster complex.  $w_U$ clearly identifies straight lines with straight lines, and it follows from \S \ref{clustercomplex} that it also identifies the relevant maximally broken lines.  Thus, $w_U$ preserves convexity along broken lines.

Note that Lemmas \ref{DecreasingBroken} and \ref{ConvDec} imply that the valuation functions $\val_v$ are convex along broken lines.  From the proof of Theorem \ref{symdual}, $\vartheta_q^{\trop}=\val_{w_U(q)}$, so tropical theta functions are also convex along broken lines.  We will see this in another way below.

\begin{thm}\label{IndeTropical}
If $\varphi:U^{\trop}\rar \bb{R}$ is piecewise linear, then $\varphi$ is convex along broken lines if and only if it is convex with respect to every seed. 
The tropical functions on $V^{\trop}$ are exactly the tropicalizations of regular functions on $V$, and the indecomposable tropical functions are exactly the tropicalizations of theta functions.
\end{thm}
\begin{proof}
In $U^{\trop}$ (hence $V^{\trop}$) with its canonical integral linear structure, broken lines can only bend away from the origin.  Let $\f{L}^d$ denote a fiber $\varphi=d$ for some piecewise linear function $\varphi$.  $\varphi$ being convex means that when $d<0$, $\f{L}^d$ only bends towards the origin, and when $d>0$, $\f{L}^d$ only bends away from the origin.  Locally changing to an affine structure in which some broken line is straight will only cause lines to bend more {\it towards} the origin.   Thus, on a cone where $\varphi$ is non-positive, convexity of $\varphi$ along broken lines is equivalent to convexity along straight lines. 

Now, suppose that $\varphi$ is convex along straight lines and non-negative on some (necessarily convex) cone $\sigma$.  We saw in \S \ref{clustercomplex} that $\sigma$ must live in the cluster complex.  Convexity of $\varphi$ along broken lines is now equivalent to convexity along the broken lines which take the maximal allowed bend across each ray in $\sigma$.  Any such broken line lives in some $\f{L}_q^{d>0}$, and it follows from Corollary \ref{FiniteScatter} that there is a seed for which $\f{L}_q^{d>0}$ is straight in the corresponding linear structure.

In summary, convexity of $\varphi$ along broken lines is equivalent to convexity along straight lines in $V^{\trop}$ and along maximally broken lines in the cluster complex.  Any maximally broken line in the cluster complex is straight with respect to some seed, and the same is locally true for straight lines in $V^{\trop}$.  Thus, convexity with respect to every seed implies convexity along broken lines.  On the other hand, every line which is straight with respect to some seed is a broken line, so convexity along broken lines implies convexity with respect to every seed. 

Now, given any regular function $f$ on $V$, we know that the restriction of $f$ to any seed torus is regular, and so $f^{\trop}$ is convex with respect to any seed.  This gives an alternative proof of the fact that tropicalizations of regular functions are convex along broken lines.

Now suppose that $\varphi$ is not an indecomposable tropical function.  Then $\varphi = \min(f_1,f_2)$ for two tropical functions $f_1$ and $f_2$, neither of which is globally equal to $\varphi$.  So we can find cones $\sigma_1,\sigma_2$ sharing a boundary ray $\rho$ such that $\varphi|_{\sigma_i} = f_i$ and $f_1(x)\neq f_2(x)$ for $x\in \sigma_2$.  Suppose that $\varphi$ is linear across $\rho$ along some broken line $\gamma$ crossing $\rho$.  Since $f_1$ and $f_2$ are both convex along this broken line and $\varphi$ is their minimum, they must both be equal to $\varphi$ in a neighborhood of $\rho$.  This contradicts our assumption that $f_1(x)\neq f_2(x)$ for $x\in \sigma_2$, so $\varphi$ must bend across $\rho$ along $\gamma$.

When crossing from $\sigma_1$ to $\sigma_2$ above, $\varphi$ changed from $f_1$ to $f_2$.  If we continue going around the origin in the same direction, $\varphi$ must eventually change back to $f_1$ after crossing some ray (or else it would be identically equal to $f_2$), and so we find that there are in fact at least two rays $\rho_1$ and $\rho_2$ in $V^{\trop}$ such that $\varphi$ bends nontrivially across $\rho_i$ along any broken line crossing $\rho_i$, for each $i$.

If $\vartheta_q^{\trop}$ is non-positive everywhere, then there is only one ray across which $\vartheta_q^{\trop}$ bends nontrivially along straight lines.  On the other hand, if $\vartheta_q^{\trop}$ is positive somewhere, then $\vartheta_q^{\trop}$ bends across straight lines only in the interior of $\sigma_{q_-,q_+}$, but it does not bend along $\f{L}_q^{d>0}$, which crosses any ray in the interior of $\sigma_{q_-,q_+}$.  Thus, the tropicalization of any theta function is indecomposable.

Now suppose that $\varphi$ is a tropical function such that $\varphi|_{\sigma}=\vartheta_q^{\trop}|_{\sigma}$ for some $q\in U^{\trop}(\bb{Z})$ and some open $\sigma\subset V^{\trop}$.  Then $\varphi \leq \vartheta_q^{\trop}$ everywhere in $V^{\trop}$ because there is some seed or branch cut with respect to which $\vartheta_q^{\trop}$ is linear and $\varphi$ is convex, hence equal to the minimum of its linear parts.  Thus, to show that $\varphi$ is a $\min$ of tropical theta functions, we just have to show that for every domain of linearity $\sigma$ of $\varphi$, there is some $q$ such that $\varphi|_{\sigma}=\vartheta_q^{\trop}|_{\sigma}$.

Let $\sigma_+$ be the (possibly empty) cone on which $\varphi>0$.  Let $\sigma\subset \sigma_+$ be an open subcone on which $\varphi$ is linear.  Convexity of $\varphi$ along broken lines implies that $\sigma$ contains no scattering rays.  We can choose a covariantly constant integral section section $q_{\sigma}$ of $T\sigma$ such that $\varphi(p)=p\wedge [q_{\sigma}(p)]$, where $p$ is being identified with a vector in $T_p V^{\trop}$.  If we view the fiber $F_d:=\{\varphi|_{\sigma}=d>0\}$ as part of a broken line with $q_{\sigma}$ giving the negative tangent direction, then we can extend $F_d$ indefinitely each direction, taking the maximal allowed bend at each wall it crosses to get a broken line $\f{L}_q^{d>0}$.  Then $\varphi|_{\sigma}$ is equal to $\vartheta_q^{\trop}|_{\sigma}$.

Now let $\sigma$ be a cone outside of $\sigma_+$ on which $\varphi$ is linear.  We have a fiber $F_d:=\{\varphi|_{\sigma}=d<0\}$ as before, and extending indefinitely in either direction (without bends this time) gets us a line $L_{q'}^{d<0}$ containing $F_d$.  If $L_{q'}^{d<0}$ does not wrap, then we immediately see $\varphi|_\sigma=\vartheta_{q'}^{\trop}|_{\sigma}$.  If $L_{q'}^{d<0}$ does wrap, we must check that we do not actually have $\varphi|_\sigma>\vartheta_{q'}^{\trop}|_{\sigma}$.  This would contradict the convexity of $\varphi$ since $\vartheta_{q'}^{\trop}|_{\sigma}$ is obtained by developing $\varphi$ linearly along $L_{q'}^{d<0}$ as it wraps around the origin.

We have thus shown that every tropical function $\varphi$ is equal to the $\min$ of a collection of tropical theta functions.  Thus, the tropical functions are indeed all tropicalizations of regular functions, and the indecomposable tropical functions are exactly the tropical theta functions.
\end{proof}

\begin{rmk}\label{NegativeTropical}
In the above theorem, we assumed $U$ was positive.  However, we can easily extend to the negative cases.  In the negative definite cases (those where $H:=(D_i\cdot D_j)$ is negative definite), convex (along straight lines) functions on $U^{\trop}$ must be positive everywhere on $U^{\trop}_0$.  But for a positive function to be convex along broken lines, it must take the maximal possible bend along every broken line it passes.  Since $U^{\trop}$ in these cases contains infinitely many scattering rays, this is impossible.  So there are no non-trivial tropical functions on $U^{\trop}$ in these cases.  This is what we expect since there are no non-constant regular functions on $U$ in these cases.

In the strictly semi-definite cases, there are straight lines in $U^{\trop}$ which are circles, and these give fibers for a ray's worth of tropical functions. 
 In general, $U$ for these cases might not admit non-constant regular functions, so the theorem as stated does not quite hold here.  However, it is possible to deform such a $U$ to a surface admitting an elliptic fibration over $\bb{A}^1$, and powers of the fibration map give the desired regular functions.  So the theorem does hold up to deformation of $U$.
\end{rmk}

\begin{cor}\label{wLinear}
The identification $w_U:U^{\trop}(\bb{Q})\rar V^{\trop}(\bb{Q})$ extends to an integral linear isomoprhism $w_U:U^{\trop}\rar V^{\trop}$.
\end{cor}
\begin{proof}
$w_U$ extends to an integral linear function because it pulls back tropical functions (restricted to the rational points) to tropical functions (restricted to the rational points).  It is an isomorphism because $w_V$ gives the inverse map.
\end{proof}

The notions of convexity along broken lines and convexity with respect to every seed make sense in more general situations related to cluster varieties (cf. \cite{GHKK} and \cite{FG1}, respectively).

\begin{conj}\label{equivconvex}
Convexity along broken lines is always equivalent to convexity with respect to every seed.
\end{conj}

The key to proving this conjecture in dimension $2$ was Lemma \ref{model}, which says that the following conjecture holds in dimension $2$:
\begin{conj}
If $\varphi$ is a tropical function on the tropicalization of a cluster variety (or a fiber of a cluster variety) $\s{Y}$, and if $\varphi$ is positive at some point on a scattering wall $\f{w}$, then the wall-crossing formula for $\f{w}$ is the formula for some mutation in some cluster structure on $\s{Y}$  (i.e., $\f{w}$ lives in the cluster complex for some cluster structure on $\s{Y}$).
\end{conj}

The following conjecture is from \cite{GHKK}:
\begin{conj}[\cite{GHKK}]
The tropicalization of any regular function on any log Calabi-Yau variety is convex along broken lines.
\end{conj}
Proving Conjecture \ref{equivconvex} would immediately imply this, because globally regular functions are of course regular on each seed torus, and they therefore give convex functions with respect to every seed.  Of course, we also conjecture that the other parts of Theorem \ref{IndeTropical} generalize to other cluster situations (and more generally, to other log Calabi-Yau situations).

\section{Toric Constructions for Log Calabi-Yau Surfaces}\label{App}

Throughout this section, it is always possible to switch the roles of $U$ and $V$ using the symmetry of the pairing $\langle \cdot,\cdot\rangle$.  We will therefore only define and prove things for one side.

\subsection{Dual Cones}\label{DualCones}

\begin{dfn}
Let $\sigma\subseteq V^{\trop}$ be a closed cone of any dimension.
The dual cone to $\sigma$ is $\sigma^{\vee}:=\{ q \in U^{\trop}|\langle q,v\rangle \geq 0 \mbox{ for all } v\in \sigma\}$.
\end{dfn}

Let $(Z_{\sigma}, \f{D}_{\sigma})$ be the partial compactification of $V$ which includes the non-nodal points of $\f{D}_{\rho}$ for each boundary ray $\rho$ of  $\sigma$, along with the point $\f{D}_{\rho}\cap \f{D}_{\rho'}$ if $\sigma$ is two-dimensional.  From the definitions, we have:
\begin{lem}\label{reg}
$q\in \sigma^{\vee}\cap U^{\trop}(\bb{Z})$ if and only if the global regular function $\vartheta_q$ on $V$ extends to a regular function on $(Z_{\sigma},\f{D}_{\sigma})$.
\end{lem}

\begin{prop}\label{ConeAffine}
Let $A_{\sigma}$ be the subalgebra of $\Gamma(V,\s{O}_V)$ generated by the $\vartheta_q$'s with $q\in \sigma^{\vee}$.  Then $A_{\sigma}=\Gamma(Z_{\sigma},\s{O}_{Z_{\sigma}})$.  Furthermore, $\{\vartheta_q\}_{q\in \sigma^{\vee}}$ forms an additve basis for $A_{\sigma}$.  If $\sigma^{\vee}$ is two-dimensional, then $\Spec A_{\sigma}$ is obtained from $Z_{\sigma}$ by contracting all the $(-1)$ curves which intersect $\f{D}_{\sigma}$.
\end{prop}
\begin{proof}
The first statement is immediate from Lemma \ref{reg}.  For the second, note that if $\langle q_i,v\rangle \geq 0$ for $i=1,2$, then $\val_v(\vartheta_{q_1}\vartheta_{q_2})\geq 0$.  Thus, a product of any $\vartheta_{q_i}$'s with $q_i \in \sigma^{\vee}$ is a sum of $\vartheta_q$'s with $q\in \sigma^{\vee}$, and the second statement follows.

For the final statement, first note that since the $(-1)$-curves are complete, all global functions are constant along them, so their contraction does not change the ring of global functions.  On the other hand, taking $\Spec$ of the global functions does not contract the boundary divisors since there {\it are} non-constant functions along them corresponding to $\vartheta_q$ with $q$ in the boundary of $\sigma^{\vee}$.  Now, we can find a global function $f\in A_{\sigma}$ whose zeroes consist of the boundary divisors and the contracted $(-1)$ curves, so the rest of $Z_{\sigma}$ is the distinguished open subset $D(f)$.  Thus, the described $(-1)$-curves are all that gets contracted from taking $\Spec$ of the global sections.\footnote{A previous version of this paper incorrectly claimed that $\Spec A_{\sigma} = Z_{\sigma}$ when $\sigma^{\vee}$ is two-dimensional, neglecting to contract the $(-1)$-curves.}
\end{proof}

This resembles the well-known construction of toric varieties from fans presented in \cite{Fult}.  Unfortunately, the contracted $(-1)$-curves prevent us from actually being able to construct compactifications of $V$ from fans in quite the way one does with toric varieties.  Furthermore, the condition that $\sigma^{\vee}$ is two-dimensional is rather strong---for example, if all lines in $U^{\trop}$ wrap, then $\sigma^{\vee}$ is never two-dimensional unless $\sigma = \{0\}$.

\subsection{Polytopes}\label{BundlePolygons}

\begin{dfn}\label{ConvexityDefinitions}
Let $Q$ be any subset of $V^{\trop}$.  The {\it polar polytope} $Q^{\circ}$ is the set $\{q\in U^{\trop}|\langle q,v\rangle \geq -1 \mbox{ for all } v\in Q\}$. 

The {\it strong convex hull}\footnote{It follows from Theorem \ref{symdual} that this is equivalent to the version of convex hull used in \cite{Shen}.  Similarly for the Minkowski sums of \S \ref{Minkowski}.  Also, Proposition \ref{ConvexCondition} below implies that our notion of strongly convex is equivalent to \cite{FG2}'s notion of convex.} of a set $Q\subset U^{\trop}$ is the set 
\begin{align*}
\Conv(Q) = \left\{ x \in U^{\trop}|\langle x,v\rangle \geq \inf_{q\in Q} \langle q,v\rangle \mbox{ for all } v\in V^{\trop}\right\}.
\end{align*}

Let $f=\sum_{q\in Q} a_q \vartheta_q \in \Gamma(V,\s{O}_V)$, $a_q \neq 0$, for some finite set $Q\subset U^{\trop}(\bb{Z})$.  The {\it Newton Polytope} of $f$ is the set $\Conv(Q)$.  Equivalently, $\Newt(f) = \{x\in U^{\trop}|\langle x,v\rangle \geq f^{\trop}(v) \mbox{ for all } v\in V^{\trop}\}$.  By a {\it vertex} of $\Newt(f)$, we mean a point $q\in Q$ as above such that $\Conv(Q\setminus \{q\}) \neq \Conv(Q)$.

A set $Q$ is called {\it strongly convex} if $Q=\Conv(Q)$.
\end{dfn}

The following lemma follows directly from the definitions.
\begin{lem}\label{PolarLemmas}
For any set $Q\subseteq V^{\trop}$, $Q\subseteq (Q^{\circ})^{\circ}$.  If $Q\subseteq S$, then $S^{\circ} \subseteq Q^{\circ}$.
\end{lem}

\begin{dfn}
We call a polytope $Q$ {\it polar} if $Q=(Q^{\circ})^{\circ}$.
\end{dfn}

\begin{lem}
$Q^{\circ}$ is polar.  Thus, $Q$ being polar is equivalent to $Q$ being the polar polytope of some set.
\end{lem}
\begin{proof}
The first statement of Lemma \ref{PolarLemmas} immediately gives us $P^{\circ}\subseteq ((P^{\circ})^{\circ})^{\circ}$.  It also gives us $P\subseteq (P^{\circ})^{\circ}$, and then the second statement gives us $((P^{\circ})^{\circ})^{\circ} \subseteq P^{\circ}$.
\end{proof}

\begin{prop}\label{ConvexCondition}
A set $Q\subset U^{\trop}$ is strongly convex if and only if it is an intersection of sets of the form $\{\langle \cdot,v\rangle \geq a_v\}$.
\end{prop}
\begin{proof}
$\Conv(Q)$ is by definition an intersection of sets of this form, with $a_v:=\inf_{q\in Q} \langle q,v\rangle$.  So $Q$ being strongly convex implies it has this form.

Conversely, suppose $Q=\bigcap_{v\in I} \{q\in U^{\trop}|\langle q,v\rangle \geq a_v\in \bb{R}\}$.  If $Q$ is not convex, then there is some $x\notin Q$ such that for every $v\in V^{\trop}$, $\langle x,v\rangle \geq \langle q_v,v\rangle$ for some $q_v\in Q$ (since $Q$ is closed, the infimum in the definition of $\Conv(Q)$ is obtained for some $q_v\in Q$).  But this implies $x$ is in each of the sets in the intersection defining $Q$, hence in $Q$.
\end{proof}

\begin{cor}\label{dualconvex}
Polar polytopes are exactly the strongly convex polytopes containing the origin in their interiors.  These are the same as the ordinary convex polytopes with the origin in their interiors.
\end{cor}
\begin{proof}
Polar polytopes by definition have the form given in Proposition \ref{ConvexCondition}.  So polar polytopes are convex.  It is easy to see that they contain $0$ in their interiors.

Conversely, strongly convex polytopes with $0$ in their interiors have the form given in Proposition $\ref{ConvexCondition}$ with each $a_v < 0$.  Thus, by multiplying the $v$'s by positive scalars, we can assume each $a_v$ equals $-1$.  The form from Proposition \ref{ConvexCondition} is then the definition for a polar polytope.

For the last statement, sets of the form $\{\langle \cdot,v\rangle \geq a_v\}$ with $a_v < 0$ are exactly the zero-sides of straight lines in $U^{\trop}$, and ordinary convex polytopes (i.e., those which are convex with respect to the canonical integral linear structure on $U^{\trop}$) with the origin in their interiors are the intersections of such sets.
\end{proof}

Recall our notation $Q_{\varphi}= \{q\in U^{\trop}|\varphi(q)\geq -1\}$, for $\varphi$ a piecewise linear function on $U^{\trop}$.  We use the analogous notation in $V^{\trop}$.

\begin{prop}\label{TropSelfDual}
If $\varphi:V^{\trop}\rar \bb{R}$ is tropical, then $Q_{\varphi}$ is polar.  If $\varphi$ is integral piecewise linear and $Q_{\varphi}$ is polar and bounded, then $\varphi$ tropical.
\end{prop}
\begin{proof}
First suppose that $\varphi$ is tropical.  Lemma \ref{PolarLemmas} gives us $Q\subseteq (Q^{\circ})^{\circ}$. On the other hand, Theorem \ref{IndeTropical} tells us that there is some regular function $f$ on $V$ with $f^{\trop}=\varphi$.  We can write $f=\sum_{q\in S} a_q \vartheta_q$, $a_q\neq 0$ for some finite set $S\subset U^{\trop}(\bb{Z})$.  Since $f^{\trop}(v)=\min_{q\in S} \langle q,v\rangle$ and $v\in Q$ if and only if $f^{\trop}(v) \geq -1$, this means that $S \subseteq Q^{\circ}$.  Now, $v\in (Q^{\circ})^{\circ}$ means that $\langle q,v\rangle \geq -1$ for all $q \in Q^{\circ}$, hence all $q \in S$, and this implies that $f^{\trop}(v) \geq -1$.  This means that $v \in Q$, as desired.

  On the other hand, $Q_{\varphi}$ being strongly convex and having the form $\{\varphi\geq -1\}$ for $\varphi$ integral piecewise linear means that it has the form $\bigcap_{q\in S} \{\langle q,\cdot\rangle \geq -1\}$ for some finite set $S\subset U^{\trop}(\bb{Z})$.  Let $f=\sum_{q\in S} \vartheta_q$.  Then $Q_{\varphi}=Q_{f^{\trop}}$.  $Q_{\varphi}$ bounded implies that $\varphi<0$ everywhere on $U^{\trop}_0$, so $Q_{\varphi}$ determines $\varphi$.  Hence, $\varphi=f^{\trop}$.
\end{proof}

Recall that in the usual vector space situation, a polytope being convex means that any line segment with endpoints in the polytope is entirely contained in the polytope.  The following theorem generalizes that characterization.
\begin{thm}\label{StrongConvexLines}
If a set $Q\subseteq U^{\trop}$ is strongly convex, then every broken line segment with endpoints in $Q$ is contained entirely within $Q$.  Conversely, if $Q$ is a rational polytope containing every broken line segment\footnote{Here we must include broken lines through the origin, by which we mean limits of sequence of broken lines whch are all equivalent to eachother in the sense of \S \ref{theta}.  In other words, in addition to the usual broken lines, we allow sets of the form $\langle \cdot,v\rangle = 0$ for $v\in V^{\trop}(\bb{Z})$.  These are the lines $L_q^0$ from Definition \ref{ZeroLine}.} with endpoints in $Q$, then $Q$ is strongly convex.  
\end{thm}
\begin{proof}
Suppose $Q$ is strongly convex.  So $Q$ is an intersection of sets of the form $\{ \langle \cdot,v\rangle \geq a_v\in \bb{R}\}$.  Let $\gamma$ be a segment of a broken line with endpoints in $Q$.  We know that each $\langle \cdot,v\rangle$ is convex along $\gamma$, so if we give $U^{\trop}$ a linear structure in which $\gamma$ is straight, then the usual notion of convexity tells us that indeed $\gamma \subset \{ \langle \cdot,v\rangle \geq a_v\in \bb{R}\}$.  Thus, $\gamma \subset Q$.

Now suppose that $Q$ is a rational polytope and that every broken line with endpoints in $Q$ is contained entirely within $Q$.  Assume that $Q$ is two-dimensional (the lower dimensional cases are easier).  We claim that the boundary of $Q$ is a finite union of closed sets $\Gamma$ each of which satisfies $\langle \Gamma, v_{\Gamma}\rangle = a_\Gamma\in \bb{Q}$ for some $v_{\Gamma}\in V^{\trop}(\bb{Z})$ such that $\langle q,v_{\Gamma}\rangle \geq a_{\Gamma}$ for all $q\in Q$.  This implies that $Q=\bigcap_{\Gamma} \{q\in U^{\trop}| \langle q,v_{\Gamma} \rangle \geq a_{\Gamma}\}$, which by Proposition  \ref{ConvexCondition} means that $Q$ is strongly convex.

It is not hard to see that each point of the boundary is contained in a closed interval $\Gamma$ (of length $>0$) which can be extended to a fiber $\wt{\Gamma}=\{\langle \cdot, v_{\Gamma}\rangle = a_v\}$ for some $v\in V^{\trop}(\bb{Z})$,  $a_v\in \bb{Q}$, satisfying $\langle q,v\rangle > a_v$ for some $q\in Q$.  Suppose there is also a $q'\in Q$ such that $\langle q',v\rangle = a_v'< a_v$.  Since $Q$ is connected, we may assume $a_v' = a_v-\epsilon$ for any sufficiently small $\epsilon>0$.  We may also assume $q'$ is a rational point.  Let $p$ be a point in the interior of $\Gamma$.  If $\wt{\Gamma}$ is a straight line, then it is clear that rotating it slightly about $p$ will give a straight line connecting $p$ to $q'$ which is not contained in $Q$ in between, a contradiction.  If $\wt{\Gamma}$ is not straight, then there is some seed with respect to which it is straight, and here we can preform a similar rotation.  This proves the claim.
\end{proof}

\begin{rmk}
We note that rather than checking the above condition for every broken line, it suffices to check for broken lines which are rational fibers of $\langle \cdot,v\rangle$ for some $v\in V^{\trop}(\bb{Z})$.  Such broken lines are either straight in $U^{\trop}$ or are contained in the cluster complex and are straight with respect to some seed structure.  So it is not necessary to understand the entire scattering diagram to understand strong convexity.  Similarly for convexity of functions along broken lines.
\end{rmk}

\begin{egs}
\hspace{.01 in}
\begin{itemize}[noitemsep] 
\item Let $p\in U^{\trop}$ be the self-intersection point of some straight line $L$ that wraps once.  Then $\Conv(p) = Z(L)$ (the $0$-side of $L$).  Since a point is convex with respect to every seed, this shows that a polytope being strongly convex is stronger than being convex with respect to every seed.
\item In the cubic surface case from Example \ref{cubicdevelop}, the convex hull of a point $q\in U^{\trop}(\bb{Z})$ is the line segment connecting $0$ to $q$.  This illustrates the need for considering $L_q^0$.
\end{itemize}
\end{egs}

\subsubsection{Line Bundles and Polytopes}

Let $W=\sum a_i \f{D}_{v_i}$ be a $\bb{Q}$-divisor in a compactification $(Z,\f{D})$ of $V$, with $\f{D}_{v_i}$ being the divisor corresponding to some primitive $v_i \in V^{\trop}(\bb{Z})$, and $a_i \in \bb{Q}$.  Recall that $\varphi_{W}$ denotes the piecewise linear function on $V^{\trop}$ which takes the value $a_i$ at $v_i$ and is linear off the rays generated by the $v_i$'s.  Let $Q_{W} : = Q_{-\varphi_{W}}= \{v\in V^{\trop}|-\varphi_{W}(v)\geq -1\}$.  We note that if $-\varphi_{W}$ is non-positive (i.e., if $W$ is effective), then $Q_{W}$ is the convex hull of the points $\frac{1}{a_i} v_i$ (since $0\in Q_{W}$, convex and strongly convex are equivalent).

\begin{dfn}
$Q_{W}^{\vee}:= \{q\in U^{\trop}| \langle q,v_i \rangle \geq -a_i \mbox{ for all } i\}$.  Note that this actually depends on $W$, not just on the polytope $Q_{W}$.
\end{dfn}

It follows easily from the definitions that:
\begin{prop}\label{BundleGlobal}
\hspace{.01 in}
\begin{itemize}
\item If $W$ is integral, then $q\in Q_{W}^{\vee}\cap U^{\trop}(\bb{Z})$ if and only if $\vartheta_q\in \Gamma(Z,\s{O}(W))$.  Thus, as a vector space, $\Gamma(Z,\s{O}(W)) = \bigoplus_{q\in Q_{W}^{\vee}\cap U^{\trop}(\bb{Z})} \kk \vartheta_q$.
\item $Q_{W}^{\vee}$ is the Newton polytope of a generic section of $\s{O}(W)$.
\item If $W$ is effective, then $Q_{W}^{\vee} = Q_{W}^{\circ}$.  In general, $Q_{W}^{\vee} \subseteq Q_{W}^{\circ}$.
\item Let $f$ be a regular function on $V$.  Let $W_f$ be negative the boundary divisor corresponding to $f^{\trop}$.  That is, for a compactification with $\f{D}=\sum_{v_i} \f{D}_{v_i}$ and $f^{\trop}$ bending only along $\rho_{v_i}$'s, $W_f:=-\sum f^{\trop}(v_i) \f{D}_{v_i}$.  Then $\Newt(f)=Q_{W_f}^{\vee}$, which from above equals $Q_{W_f}^{\circ}$ if $f^{\trop}$ is negative everywhere on $U^{\trop}_0$.
 \end{itemize}
\end{prop}

\begin{prop}\label{ConvexDual}
The strongly convex integral (resp. rational) polytopes are exactly those of the form $Q_{W}^{\vee}$ for some divisor (resp. some $\bb{Q}$-divisor) $W$.  Thus, strongly convex integral polytopes are the same as Newton polytopes.
\end{prop}
\begin{proof}
The first part follows immediately from the definition of $Q_{W}^{\vee}$ and Proposition \ref{ConvexCondition}.  The second part follows because $Q_W^{\vee}$ is the Newton polytope of a generic section of $\s{O}(W)$.
\end{proof}

Using our descriptions of the fibers of $\val_{v_i}$ from Corollary \ref{level} and Proposition \ref{levelpos}, we can easily describe $Q_{W}^{\vee}$ explicitely.  In particular:
\begin{prop}\label{PolarDescription}
Use $w_U$ to identify $U^{\trop}$ with $V^{\trop}$.  Assume that $W=\sum a_i \f{D}_{v_i}$ is strictly effective (so each $a_i >0$).  Then: 
\begin{align*}
Q_{W}^{\vee} = Q_W^{\circ} = \bigcap_i \?{Z(L_{v_i}^{a_i})}.
\end{align*}
\end{prop}

This is analogous to the toric picture of a ``normal polytope,'' except that using the wedge form in place of the dot product results in ``parallel polytopes.''  This description was previously observed in \cite{GHK2}.

May other facts about polytopes from the toric world generalize to our situation.  For example:

\begin{prop}\label{PointsIntersection}
Let $W=\sum_{i=1}^n a_i \f{D}_{v_i}$ be an integral Weil divisor (subscripts cyclically ordered), and let $F_{v_i}$ be the (possibly empty) set $Q_{W}^{\vee} \cap \{q\in U^{\trop}| \langle q,v_i\rangle = -a_i\}$.  Let $d_i$ be one less than the number of points in $F_{v_i}\cap U^{\trop}(\bb{Z})$.  If $1\leq d_i < \infty$, then $W\cdot \f{D}_{v_i} \geq d_i$.  If $n\geq 2$, then $W\cdot \f{D}_{v_i} = d_i$ if and only if $F_{v_i}\cap F_{v_{i-1}}\neq \emptyset$ and $F_{v_i}\cap F_{v_{i+1}}\neq \emptyset$.
\end{prop}
\begin{proof}
Elements of $\Gamma(\f{D}_{v_i},\s{O}_{Z}(W)|_{\f{D}_{v_i}})$ correspond to rational functions $f$ on $V$ with $\val_{\f{D}_{v_i}}(f) = -a_i$ and $\val_{\f{D}_{v_{i\pm 1}}}(f) \geq -a_{i\pm 1}$.  The latter condition is satisfied by all $\vartheta_q$ with $q\in Q_W^{\vee}$, while the former is satisfied for all $\vartheta_q$ with $\langle q,v_i\rangle = -a_i$.  There are thus at least $d_i+1$ global sections of $\s{O}_{Z}(W)|_{\f{D}_{v_i}}$, so this line bundle has degree at least $d_i$.  This gives us $W\cdot \f{D}_{v_i} \geq d_i$.  Since the condition $\val_{\f{D}_{v_{i\pm 1}}}(\vartheta_q) \geq -a_{i\pm 1}$ is satisfied if and only if $\langle q,v_{i\pm 1}\rangle \geq -a_{i\pm 1}$, we see that there is indeed equality precisely when $F_{v_i}\cap F_{v_{i\pm 1}}\neq \emptyset$.
\end{proof}

\begin{cor}\label{AmpleConvex}
Any strongly convex polytope $Q$ in $U^{\trop}$ is $Q_W^{\vee}$ for some not necessarily effective $\f{D}$-ample divisor $W$ in some compactification of $V$.  $W$ is effective if and only if $Q$ contains the origin.
\end{cor}
\begin{proof}
We can write $Q=\bigcap_{v\in S} \{ \langle \cdot,v\rangle \geq -a_v\}$ with $S$ minimal.  Choose the compactifying divisor to be $\f{D}=\sum \f{D}_{v'}$, where the sum is over all primitive vectors $v'\in S$.  Then $W=\sum_{v\in S} a_v \f{D}_v$, where if $v=|v|v'$ with $v'$ primitive, then $\f{D}_v$ denotes $|v|\f{D}_{v'}$.  Since $S$ is minimal, if $\partial Q$ contains at least two integral points then each $F_v$ must contain at least two integral points, and Proposition \ref{PointsIntersection} implies that $W$ is $\f{D}$-ample.  If $\partial Q$ only contains one integral point, then it contains only one edge, and so the corresponding $\f{D}$ is irreducible, hence ample by our positivity assumption.
\end{proof}

\cite{GHK2} describes the corresponding maps to projective space in the cases where the $\f{D}$-ample divisor $W$ is effective.  
 We do not need $W$ to be effective because we do not require the origin to be in the interior of the stongly convex polytope.  However, $\f{D}$-ample divisors which are not effective are typically not ample on $V$, even if $V$ is generic.

\begin{eg}\label{CubicReflexive}
For $U$ the affine cubic surface from Example \ref{cubicdevelop}, $U^{\trop}$ contains a reflexive polytope $Q$ which includes four integral points---with respect to the coordinates used for the first patch in Example \ref{cubicdevelop}, these points are the origin, $(1,0)$, $(0,1)$, and $(-1,1)$.  The corresponding divisor $\f{D}=\f{D}_1+\f{D}_2+\f{D}_3$ compactifying $V$ is $\f{D}$-ample by Corollary \ref{AmpleConvex}, hence ample in $Z:=V\cup \f{D}$ since it is also strictly effective.  Furthermore $\f{D}^2=3$ by Proposition \ref{PointsIntersection}.  Thus, $(Z,\f{D})$ is a degree $3$ del Pezzo surface, hence a cubic surface.  Since the only information we used about $U$ was that its tropicalization was $U^{\trop}$, this shows that any $U'$ whose tropicalization is $U^{\trop}$ has a cubic surface as (a compactification of) its mirror.  So then, since $U'$ can (up to deformation or contraction of interior $(-2)$-curves) be identified with a fiber of the mirror by \cite{GHK2}, $U'$ is itself an affine cubic surface.  I.e., in this case, $U^{\trop}$ determines $U$ up to deformation.  A similar argument applies whenever $U^{\trop}$ contains a reflexive polytope $Q$.
\end{eg}

\begin{eg}
Suppose $\Newt(f)$ is two-dimensional and is given by $\bigcap_{v\in I} \{\langle \cdot,v\rangle \geq a_v\}$ with $I$ being minimal in the sense that removing some $v$ would result in the intersection being a larger set.   Then $r$ being a vertex (recall the definition of a vertex from Definition \ref{ConvexityDefinitions}) means that it is a point on the boundary where $\{\langle \cdot,v\rangle = a_v\}$ intersects $\{\langle \cdot,v'\rangle = a_{v'}\}$ for some points $v\neq v'$ in $I$.  In case $I$ has only one element, $r$ can be a self-intersection point of $\{\langle \cdot,v\rangle = a_v\}$.  Vertices, however, do not include points that only look like vertices because they are kinks in some broken line (after all, such points no longer look like vertices when viewed with respect to some seed).
\end{eg}

\subsection{Tropical Multiplication and Minkowski Sums}\label{Minkowski}

The theta function multiplication formula in Theorem \ref{thetamult} is quite complicated, potentially involving very large numbers of broken lines.  However, tropicalization allows us to at least see which theta functions might have nonzero coefficients in a product $\vartheta_{q_1} \vartheta_{q_2}$.  If $f=\sum c_q \vartheta_q$ is a regular function, then $c_q=0$ unless $q\in \Newt(f)$, and if $q$ is a vertex 
 of $f$, then $c_q\neq 0$.  We would therefore like to describe $\Newt(\vartheta_{q_1}\vartheta_{q_2})$.

\begin{dfn}
The {\it Minkowski sum} of a collection of strongly convex integral polytopes $\Newt(f_i)$, $i=1,\ldots,s$, is $\Newt(f_1)+\ldots+\Newt(f_n):=\Newt(f_1\cdots f_n)$.

More generally,\footnote{This more general definition for the Minkowski sum was previously used in \cite{FG2} and \cite{Shen}.} the Minkowski sum of a set of strongly convex (not necessarily integral) polytopes $Q_i:=\bigcap_{v\in V^{\trop}(\bb{Z})} \langle \{x\in U^{\trop}|\langle q,v\rangle \geq a_{v,i}\in \bb{R}\}$ is $Q_1+\ldots+Q_s:=\bigcap_{v\in V^{\trop}(\bb{Z})}\left\{x\in U^{\trop} | \langle x,v\rangle \geq a_{v,1}+\ldots + a_{v,s} \right\}$.
\end{dfn}

To see the equivalence of the definitions in the case of integral polytopes, recall that $(fg)^{\trop}=f^{\trop}+g^{\trop}$, so $\Newt(fg) = \{x\in U^{\trop}|\langle x,v\rangle \geq f^{\trop}(v)+g^{\trop}(v) \mbox{ for all } v\in V^{\trop}\}$.

\begin{eg}
It is immediate from the definition and the fact that $\langle \cdot,\cdot\rangle$ respects scaling that,
for any $k \in \bb{Z}_{\geq 0}$ and any strongly convex polytope $Q$, $\sum_{i=1}^k Q = kQ:= \{ ku|u \in Q\}$.  In particular, $\Newt(f^k) = k\Newt(f)$.
\end{eg}
Finding a nice formula for $\Newt(f_1,\cdots,f_s)$ in general is a bit more complicated due to the fact that the monodromy in $U^{\trop}$ prevents addition from being well-defined.  We will assume that the $f_i^{\trop}$'s are all non-positive (i.e., their Newton polytopes contain $0$).

\begin{thm}\label{MinkowskiSum}
Let $Q_1,\ldots,Q_s\subset U^{\trop}$ be strongly convex integral polytopes containing the origin.   Let $\rho_1,\ldots,\rho_m$ be a collection of rays in $U^{\trop}$ not intersecting the vertices of the $Q_k$'s such that no two non-equal vertices from different $Q_k$'s lie in the same component of $U^{\trop}\setminus \bigcup_{i=1}^m \rho_i$.  Then
\begin{align}\label{SumFormula}
Q_1+\ldots +Q_s = \Conv \left(\bigcup_{i=1}^m \left(Q_1+_i \ldots +_i Q_s\right)\right),
\end{align}
where $+_i$ denotes addition on the complement of $\rho_i$, and $Q_1+_i \ldots +_i Q_s:=\{q_1+_i\ldots +_i q_s \in U^{\trop}| q_k \in Q_k\}$.\footnote{If some $q_1+_i\ldots +_i q_s$ is not defined in $U^{\trop}$, we simply do not include it in the set.  Alternatively, we could only include $q_1+_i \ldots +_i q_s$ in the set if there is some convex cone $\sigma$ in the complement of $\rho_i$ containing each $q_k$---addition in a convex cone is always well-defined.}
\end{thm}
\begin{proof}

Fix a generic $v\in V^{\trop}$, and choose some generic $f_i  = \sum a_{q,i} \vartheta_{q}$ for each $i$ so that $Q_i=\Newt(f_i)$.  Define $f:=\prod f_i$.  Then \begin{align*}
f^{\trop}(v)=\min_S (\prod_{q\in S} \vartheta_q)^{\trop}(v) = \min_S \sum_{q\in S} \vartheta_{q}^{\trop}(v),\end{align*} where the minimum is over all sets $S$ containing exactly one vertex of each $Q_i$.  

 The containment $\supseteq$ is easy, since by the multiplication formula in Theorem \ref{thetamult}, each vertex of the right-hand side of Equation \ref{SumFormula} corresponds to a theta function showing up in the expansion of some $\prod_{q\in S} \vartheta_q$.

For the other direction, it suffices to consider only $S$'s such that $\vartheta_q^{\trop}(v) \leq 0$ for each $q\in S$.  Fix one such $S$.  For some choice of ray $\rho:=\rho_{S,v}$ and $j\in \bb{Z}_{\geq 0}$, each $L_q^{d<0}$ with $q\in S$ crosses $\rho$ $j$ times before the $L_q^{d<0,0}$ part crosses $\rho_v$.  Then as in the proof of Lemma \ref{tpos}, we have $\vartheta_q^{\trop}(v)= v\wedge_{\rho_{S,v}} \mu^j(q)$ for each $q\in S$.  So for $S=\{q_1,\ldots,q_s\}$, we have 
\begin{align*}
\sum_{q\in S} \vartheta_{q}^{\trop}(v) &= \sum_{q\in S} (v\wedge_{\rho} \mu^j(q)) \\
                                       &= v\wedge_{\rho} \mu^{j}(q_1+_{\rho} \ldots +_{\rho} q_s) = \vartheta_{q_1+_{\rho}\ldots+_{\rho} q_s}^{\trop}(v).
\end{align*}

So now we have $f^{\trop}(v)= \min_{S=\{q_1,\ldots,q_s\}} \vartheta_{q_1+_{\rho_{S,v}}\ldots +_{\rho_{S,v}} q_s}^{\trop}(v)$, where we are now assuming that $\vartheta_{q_i}^{\trop}(v)\leq 0$ for each $q_i\in S$.  Thus, 
\begin{align*}
Q_1+\ldots + Q_S \subseteq \Conv \left(\bigcup_{S,v} \left(Q_1+_{\rho_{S,v}} \ldots +_{\rho_{S,v}} Q_s\right)\right),
\end{align*}
and this is clearly contained in the right-hand side of Equation \ref{SumFormula}, as desired.
\end{proof}

We now give another Minkowski sum formula which is perhaps more elegant and easier to prove directly, but requires more work to apply.  Consider the universal cover $\xi:\wt{U}_0^{\trop} \rar U^{\trop}_0$ with the pulled back integral linear structure.  Consider a collection of points $q_1,\ldots,q_s\in \wt{U}_0^{\trop}$.  If these points live in some convex cone $\sigma \subset \wt{U}_0^{\trop}$, then they can be added together in an obvious and canonical way, and we say that the sum is {\it well-defined}.

\begin{dfn}
Given a collection of subsets $\wt{Q_1},\ldots,\wt{Q_s} \subseteq \wt{U}_0^{\trop}$, define 
\begin{align*}
\wt{Q_1}+\ldots+\wt{Q_s}:=\left\{q_1+\ldots+q_s|q_i\in \wt{Q_i} \mbox{ and $q_1+\ldots+q_s$ is well-defined}\right\}.
\end{align*}
\end{dfn}

\begin{cor}
Let $Q_1,\ldots,Q_s\subset U^{\trop}$ be strongly convex integral polytopes containing the origin.  Let $\wt{Q_i}:=\xi^{-1}(Q_i\setminus \{0\})$.  Then
\begin{align}\label{UniversalSumFormula}
Q_1+\ldots+Q_s = \xi(\wt{Q_1}+\ldots+\wt{Q_s})\cup \{0\}.
\end{align}
\end{cor}
\begin{proof}
The containment $\supseteq$ follows easily as in the proof of Theorem \ref{MinkowskiSum}.  For the other containment, it is easy to see that every sum from the right-hand side of Equation \ref{SumFormula} also comes from a sum in the right-hand side of Equation \ref{UniversalSumFormula}.
\end{proof}

We say that $U$ is of {\it finite type} if no lines in $U^{\trop}$ wrap---equivalently, if the cluster complex includes all of $U^{\trop}$.  See \cite{Man2} for several other equivalent characterizations, including the corresponding cluster structures being finite type in the usual sense for cluster varieties.  The following Minkowski sum formula was proven for cluster varieties of type $A_n$ in \cite{Shen}.

\begin{cor}
Suppose that $U$ is of finite type.  Let $Q_1,\ldots,Q_s$ be strongly convex integral polytopes containing the origin.  Then
\begin{align}\label{SeedSum}
Q_1+\ldots+Q_k =\bigcup_E Q_1+_E+\ldots+_E Q_s,
\end{align}
where the union is over all seeds $E$, and $+_E$ denotes addition in $U^{\trop}$ with respect to the vector-space structure induced by the seed $E$.
\end{cor}
\begin{proof}
The containment $\supseteq$ is easy as in the proof of Theorem \ref{MinkowskiSum}.  For the reverse, note that for any set $S=\{q_1,\ldots,q_s\}$ contributing the right-hand side of Equation \ref{SumFormula}, we can say that there is a strictly convex cone $\sigma\supset S$ in which the addition is preformed.  By Lemma \ref{model}, there is some seed $E_{\sigma}$ in which all non-toric blowups correspond to rays in the complement of $\sigma$, so addition in $\sigma$ is the same as addition with respect to the seed $E_{\sigma}$.  Thus, the right-hand side of Equation \ref{SumFormula} is contained in the right-hand side of Equation \ref{SeedSum}, as desired.
\end{proof}

\section{Integral Formulas} \label{period}

For this section, let $\kk=\bb{C}$.  Recall that since $V$ is log Calabi-Yau like $U$, it has a holomorphic volume form $\Omega$ with log poles along the boundary $\f{D}$ of any maximal-boundary compactification $(Z,\f{D})$.  \cite{GHK2} defines a class $\gamma \in H_2(V,\bb{Z})$ as follows.  Take any nonsingular $(Z,\f{D}=\f{D}_1+\ldots+\f{D}_n)$ as above.  Then $\gamma$ is the class of a torus $0<|z_i| = |z_{i+1}| = \epsilon \ll 1$, where $z_i$ and $z_{i+1}$ are local coordinates for $Z$ in a neighborhood of $p=\f{D}_i\cap \f{D}_{i+1}$ such that $\f{D}_i$ is locally given by $z_i=0$.

\begin{lem}\label{gammaClass}
The class $\gamma$ is canonical (it does not depend on our choice of compactification or vertex $p$).  This remains true even if we remove from $Z$ a curve $C$ which intersects only one boundary divisor.
\end{lem}
\begin{proof}
Suppose we have two different choices of compactification of $V$.  Then we apply the following argument to a common toric blowup of the two:

Observe that each toric model $(Z,\f{D})\rar (\?{Z},\?{\f{D}})$ gives us a torus $T=(\bb{C}^*)^2$ in $V$, equal to the complement of the exceptional divisors in $V$ (in the language of cluster varieties, this is the corresponding seed torus).  In fact, the complement of the images of the exceptional divisors in $\?{Z}$ can be identified with a subvariety of $Z$.  It is well-known that there is a ``moment map'' from $\?{Z}$ to a polygon $Q$ in $M_{\bb{R}}$ with $\f{D}$ mapping to the boundary of the polygon and with fibers over the $k$-dimensional faces being $k$-dimensional tori in the $k$-strata of $\?{Z}$.  So each $p_i=\f{D}_i\cap \f{D}_{i+1}$ maps to a vertex $\?{p_i}$ of $Q$.  $z_i$ and $z_{i+1}$ can be chosen so that $\gamma$ is a fiber of the moment map over a point very close to $p$.  Since all the fibers are homologous, the first claim follows from taking fibers near different vertices.

Suppose we remove a curve $C$ intersecting, say, $\f{D}_i$.  Let $\?{C}$ denote the closure in $\?{Z}$ of $C\cap T$.  Then the image of $\?{C}$ under the moment map only intersects the edge $F_i$ which is the image of $\f{D}_i$.  So even on the complement of the image of $\?{C}$, there is a path in $Q$ between any two of $Q$'s vertices, showing that the claim still holds.
\end{proof}

See \cite{GHK2} for a slightly different proof of the first statement of the lemma. 

\begin{rmk} Conjecturally, $\gamma$ is the homology class of a fiber of an SYZ fibration of $V$ over $V^{\trop}$.  At the very least, if we factor the singularity in $V^{\trop}$ into focus-focus singularities which are still contained in some convex polytope $Q$, then $V$ admits a Largangian fibration over the interior of $Q$. 
 See \cite{Sym} for the details.  This fibration can be used for an alternative proof of the lemma. 
\end{rmk}

Assume $\Omega$ is normailized\footnote{Recall that if we take the cyclic ordering of $D=D_1+\ldots+D_n$ as part of our data, then we can use this to orient $U^{\trop}$, and $V^{\trop}$ gets the opposite orientation.  This can be used to orient $\gamma$ (by ordering $z_i$ and $z_{i+1}$).  Alternatively, we can take the sign of $\Omega$ as part of our data and say that $\gamma$ is oriented to make $\int_{\gamma}\Omega>0$.} so that $\int_{\gamma} \Omega = 1$.  Following \cite{GHK2}, we define a function $\Tr:\s{O}_V(V)\rar \bb{C}$,
\begin{align*}
    Tr(f) := \int_{\gamma} f \Omega.
\end{align*}
\cite{GHK2} shows that $\Tr(f)$ is equal to the coefficient of $\vartheta_0=1$ in the unique expression of $f$ as a linear combination of theta functions.  We will now describe how to modify this to give the coefficients of the other theta functions.

For $q \in U^{\trop}(\bb{Z})$, define $\Tr_q:\s{O}_V(V)_{\vartheta_q} \rar \bb{C}$ by
\begin{align*}
    \Tr_q(f) := \int_{\gamma} f \vartheta_{q}^{-1} \Omega
\end{align*}
\begin{lem}
$\Tr_q$ is well-defined. 
\end{lem}
\begin{proof}
Since $\vartheta_q^{-1}$ is only regular on $V\setminus Z(\vartheta_q)$, it is not immediately clear from Stokes' theorem that this definition is independent of our choice of $p$ for defining $\gamma$.  If $\vartheta_q^{\trop} \leq 0$ everywhere, then our description of tropical theta functions shows that the zero set $V(\vartheta_q)$ intersects only one boundary divisor, so the well-definedness follows from Lemma \ref{gammaClass}.  If $\vartheta_q^{\trop}$ is positive somewhere, then $q$ is in the cluster complex, and so there is some open torus $T$ in $V$ on which $\vartheta_q$ is a monomial and therefore has no zeroes.  The claim then follows from Lemma \ref{gammaClass} applied to $T$.
\end{proof}

\begin{lem}\label{InnerProduct}
Let $q, r\in U^{\trop}(\bb{Z})$, and suppose that $r\notin \Conv(q)\setminus \{q\}$.  Then $\Tr_r(\vartheta_q) = \delta_{q,r}$.
\end{lem}
\begin{proof}
If $r=q$, then the claim is obvious.  Otherwise, $r\notin \Conv(q)$, so there is some primitive $v\in V^{\trop}(\bb{Z})$ such that $\langle r,v\rangle < \langle q,v\rangle$.  Then $\val_{\f{D}_v}(\vartheta_q\vartheta_r^{-1}) >0$.  Since $\Omega$ only has a simple pole along $\f{D}_v$, $\vartheta_q \vartheta_r^{-1} \Omega$ is generically regular along $\f{D}_v$.  If we view $\gamma$ as the class of an $S^1$ bundle over a loop $\gamma'$ in $\f{D}_v$, then the claim follows from the Residue Theorem:
\begin{align*}
\int_{\gamma} \vartheta_q \vartheta_r^{-1} \Omega = \int_{\gamma'} \Residue_{\f{D}_v}\left(\vartheta_q \vartheta_r^{-1} \Omega\right) = \int_{\gamma'} 0 = 0.
\end{align*}
\end{proof}

\begin{thm}\label{Fourier}
Let $f=\sum_q c_{q} \vartheta_q$ be a function on $V$.  Suppose that at least one of the following hold:
\begin{itemize}[noitemsep] 
\item $r$ is not in the convex hull of any point $q\in \Newt(f)\cap U^{\trop}(\bb{Z})$ with $q\neq r$.  In particular, this includes cases where
 $r$ is a vertex of $\Newt(f)$, as well as cases where $r$ is in the complement of $\Newt(f)$.
\item $r\in U^{\trop}(\bb{Z})$ is in the cluster complex (i.e., $r=0$ or $\langle r,v\rangle >0$ for some $v$).
\end{itemize}
 Then $c_{r}=\Tr_r(f)$.
In particular, if every point of $\Newt(f)\cap U^{\trop}(\bb{Z})$ which is not a vertex is in the cluster complex, then
\begin{align}\label{FSeries}
f=\sum_{r\in U^{\trop}(\bb{Z})} \Tr_r (f) \vartheta_r.
\end{align}
\end{thm}
\begin{proof}
If $r$ is not in the convex hull of any point in $\Newt(f)\cap \left(U^{\trop}(\bb{Z})\setminus \{r\}\right)$, then the claim follows immediately from Lemma \ref{InnerProduct}.

Suppose that $r\neq 0$ is in the cluster complex.  We can refine our fan $\Sigma$ from the construction of $\s{V}$ so that there is some cone $\sigma\ni r$ which has no scattering rays on its interior.  Then there is a torus $T_{\sigma}\cong (\bb{C}^*)^2$ in $V$ corresponding to $\sigma$ on which $\vartheta_r$ is just the restriction of the monomial $z^{\wt{\varphi}(r)}$, which we may view as a constant times $z^r$.  Let $\Gamma$ be a broken line in $\sigma$ with attached monomial $z^{\wt{\varphi}(r)}$.  By flowing backwards (in the $r$ direction) along $\Gamma$, we see that $\Gamma$ does not hit any scattering walls, hence does not bend.  So $z^{\wt{\varphi}(r)}$ must have been the initial monomial attached to $\Gamma$.  Hence, $\vartheta_r$ is the only theta function whose expansion in terms of monomials in $T_{\sigma}$ contains a $z^{r}$ term.  Since $\int_{\gamma} z^q z^{-r}\Omega=\delta_{q,r}$ always holds (a standard fact about tori, and also a corollary of Lemma \ref{InnerProduct}), the claim follows.  The $r=0$ case was proven in \cite{GHK1}.
\end{proof}

\begin{eg}
In Example \ref{CubicThetas}, we saw that in the cubic surface situation with $q\neq 0$, we have $\vartheta_q^3= 3\vartheta_q+\vartheta_{3q}$, and $\vartheta_q^2=2+\vartheta_{2q}$.  Thus, \[\Tr_q(\vartheta_{3q}) = \Tr_0[(\vartheta_q^3-3\vartheta_q)\vartheta_q^{-1}] = \Tr_0(\vartheta_q^2 - 3) = \Tr_0(\vartheta_{2q}-1) = -1,\] and so Equation \ref{FSeries} fails here.
\end{eg}

\begin{rmk}
We note that Equation \ref{FSeries} resembles the formula for the Fourier series expansion of a function on a compact torus.  Indeed, in the case that $V$ is a toric variety, applying this theorem to monomials and restricting to the orbits of the torus action recovers the usual formula for finite Fourier expansions.
\end{rmk}

\begin{rmk}
Suppose that $\Newt(f)\cap U^{\trop}(\bb{Z})$ contains points which are neither vertices nor in the cluster complex.  We can still use $\Tr_q$ with various $q$ to get all the coefficients in the theta function expansion for $f$ as follows: we first use the theorem to get the coefficients for the vertices $\{q_1,\ldots,q_s\}$ of $\Newt(f)$.  We then subtract the contributions of these theta functions to get $\wt{f}:=f-\sum_{i=1}^s \Tr_{q_i}(f) \vartheta_{q_i}$.  $\Newt(\wt{f})$ is now smaller than $\Newt(f)$ (it is contained in the convex hull of $\Newt(f)\cap U^{\trop}(\bb{Z})\setminus \{q_1,\ldots,q_s\}$), so we have a new set of vertices and can apply the process again.  Repeating this will eventually yield all the coefficients.
\end{rmk}

\subsection{Theta Functions up to Scalar Multiplication}
Consider $(Z,\f{D})$, $V=Z\setminus \f{D}$, as usual.  For any regular function $f$ on $V$, let $\f{D}(f):=\sum f^{\trop}(v_i) \f{D}_{v_i}$.  Then \[f\in \Gamma[Z,\s{O}(\f{D}(f))] = \left\{\left.\sum_{q\in \Newt(f)\cap U^{\trop}(\bb{Z})} a_q \vartheta_q \right| a_q\in \bb{C}\right\}.\]  Since knowledge of $f^{\trop}$ and $U^{\trop}$ is sufficient to define $\Newt(f)$, we find that $f^{\trop}$ is often sufficient to significantly narrow down the possibilities for $f$.  We apply this in the following examples:

\begin{egs}
\hspace{.01 in}
\begin{itemize}[noitemsep] 
\item If $\Newt(f)$ is just a single point $q\in U^{\trop}(\bb{Z})$, then $f$ is uniquely determined up to scaling.  Of course, in this case, $q$ is in the cluster complex, and we have already seen an explicit description of such functions.
\item If $\Newt(\vartheta_q)\cap U^{\trop}(\bb{Z})$ is contained entirely in the cluster complex except for the point $q$, then we can identify $\vartheta_q$ as the unique (up to scaling) nonzero global section $f$ of $\s{O}(\f{D}(\vartheta_q))$ such that $\Tr_r(f)=0$ for all $r\in \Newt(\vartheta_q)\cap U^{\trop}(\bb{Z})\setminus \{q\}$.
\item One can show that for any $U^{\trop}$ with all lines wrapping, there is some $q$ with $\Conv(q)\cap U^{\trop}(\bb{Z}) = \{q,0\}$.  Then $\vartheta_q$ is uniquely determined up to scaling by $\vartheta_q^{\trop}$ and the fact that $\Tr_0(\vartheta_q) = 0$.  For example, in the cubic surface case, any primitive $q$ satisfies this condition.
\end{itemize}
\end{egs}

\bibliographystyle{modified_amsalpha}  % Here the bibliography 	
\bibliography{mandel}        % is inserted.	

\providecommand{\bysame}{\leavevmode\hbox to3em{\hrulefill}\thinspace}
\providecommand{\MR}{\relax\ifhmode\unskip\space\fi MR }
% \MRhref is called by the amsart/book/proc definition of \MR.
\providecommand{\MRhref}[2]{%
  \href{http://www.ams.org/mathscinet-getitem?mr=#1}{#2}
}
\providecommand{\href}[2]{#2}
\begin{thebibliography}{GHKK14}

\bibitem[FG06]{FG0}
V.~{Fock} and A.~{Goncharov}, \emph{{Moduli spaces of local systems and higher
  Teichm\"{u}uller theory}}, Publ. Math. IHES \textbf{103} (2006), 1--212.

\bibitem[FG09]{FG1}
\bysame, \emph{{Cluster ensembles, quantization and the dilogarithm}}, Ann.
  Sci.\'Ec. Norm. Sup. (4) \textbf{42} (2009), no.~6, 865--930.

\bibitem[FG11]{FG2}
\bysame, \emph{{Cluster $\s{X}$-varieties at infinity}}, arXiv:1104.0407.

\bibitem[{Ful}93]{Fult}
W.~{Fulton}, \emph{Introduction to toric varieties}, Annals of Mathematics
  Studies, vol. 131, Princeton University Press, Princeton, NJ, 1993.

\bibitem[GHK]{GHK2}
M.~{Gross}, P.~{Hacking}, and S.~{Keel}, \emph{{Mirror symmetry for log
  Calabi-Yau surfaces II}}, (in preparation).

\bibitem[GHK15a]{GHK3}
\bysame, \emph{Birational geometry of cluster algebras}, Algebr. Geom.
  \textbf{2} (2015), no.~2, 137--175.

\bibitem[GHK15b]{GHK1}
\bysame, \emph{Mirror symmetry for log {C}alabi-{Y}au surfaces {I}}, Publ.
  Math. Inst. Hautes \'Etudes Sci. \textbf{122} (2015), 65--168.

\bibitem[GHK15c]{GHK_MLP}
\bysame, \emph{Moduli of surfaces with an anti-canonical cycle}, Compos. Math.
  \textbf{151} (2015), no.~2, 265--291.

\bibitem[GHKK14]{GHKK}
M.~{Gross}, P.~{Hacking}, S.~{Keel}, and M.~{Kontsevich}, \emph{{Canonical
  bases for cluster algebras}}, arXiv:1411.1394.

\bibitem[GPS10]{GPS}
Mark Gross, Rahul Pandharipande, and Bernd Siebert, \emph{The tropical vertex},
  Duke Math. J. \textbf{153} (2010), no.~2, 297--362.

\bibitem[GS11]{TorDeg}
M.~{Gross} and B.~{Siebert}, \emph{An invitation to toric degenerations},
  Surveys in differential geometry. {V}olume {XVI}. {G}eometry of special
  holonomy and related topics, Surv. Differ. Geom., vol.~16, Int. Press,
  Somerville, MA, 2011, pp.~43--78.

\bibitem[KS06]{KS}
M.~{Kontsevich} and Y.~{Soibelman}, \emph{Affine structures and
  non-{A}rchimedean analytic spaces}, The unity of mathematics, Progr. Math.,
  vol. 244, Birkh\"auser Boston, Boston, MA, 2006, pp.~321--385.

\bibitem[{Li}02]{Li}
J.~{Li}, \emph{A degeneration formula of {GW}-invariants}, J. Differential
  Geom. \textbf{60} (2002), no.~2, 199--293.

\bibitem[{Man}]{Man2}
T.~{Mandel}, \emph{{Classification of rank $2$ cluster varieties}},
  arXiv:1407.6241v2.

\bibitem[{She}14]{Shen}
L.~{Shen}, \emph{Stasheff polytopes and the coordinate ring of the cluster
  {$\s{X}$}-variety of type {$A_n$}}, Selecta Math. (N.S.) \textbf{20} (2014),
  no.~3, 929--959.

\bibitem[{Sym}03]{Sym}
M.~{Symington}, \emph{Four dimensions from two in symplectic topology},
  Topology and geometry of manifolds ({A}thens, {GA}, 2001), Proc. Sympos. Pure
  Math., vol.~71, Amer. Math. Soc., Providence, RI, 2003, pp.~153--208.

\end{thebibliography}
\index{Bibliography@\emph{Bibliography}}

\end{document}